\documentclass [a4paper, twoside, 11pt]{article}
\usepackage[utf8]{inputenc}
\usepackage[T1]{fontenc}
\usepackage[english]{babel}
\usepackage{enumitem}

\usepackage[draft,english]{fixme}
\usepackage{cite}
\usepackage{yfonts}
\usepackage{mathtools}
\usepackage{amsmath, amssymb}
\usepackage{dsfont}
\usepackage[left=3cm,top=4cm,right=3cm,bottom=5cm]{geometry}
\usepackage[amsmath,thmmarks]{ntheorem} 
\theorembodyfont{\normalfont\slshape}
\theoremseparator{}
\theoremstyle{break}
\newtheorem{thm}{Theorem}[section]

\newtheorem{rem}[thm]{Remark}

\newtheorem{lemma}[thm]{Lemma}

\newtheorem{prop}[thm]{Proposition}

\theorembodyfont{\normalfont\itshape}
\newtheorem{defn}[thm]{Definition}

\theoremstyle{nonumberplain}
\theoremheaderfont{%
\normalfont\scshape}
\theorembodyfont{\normalfont}
\theoremsymbol{\ensuremath{\blacksquare}}
\theoremseparator{:}
\newtheorem{proof}{Proof}

\usepackage{keyval}
\makeatletter
\newcommand*\DeclareMathSymbolShorthand[2]{
   \begingroup
   \setkeys{DMSS}{name=#2,#1}%
   \if\DMSS@overwrite 
   \else
      \expandafter\@ifdefinable\csname \DMSS@prefix\DMSS@name\endcsname{%
        \def\DMSS@overwrite{00}
      }%
   \fi%
   \if\DMSS@overwrite 
      \expandafter\@firstofone%
   \else\expandafter\@gobble\fi%
   {\protected\expandafter%
        \xdef\csname \DMSS@prefix\DMSS@name \endcsname{%
        \unexpanded\expandafter{\DMSS@format{#2}}}}%
   \endgroup}
\define@key{DMSS}{format}{\let\DMSS@format#1}
\define@key{DMSS}{format*}{\def\DMSS@format{\expandafter#1\@firstofone}}
\define@key{DMSS}{name}{\def\DMSS@name{#1}}
\define@key{DMSS}{prefix}{\def\DMSS@prefix{#1}}
\define@key{DMSS}{overwrite}[true]{%
   \edef\DMSS@overwrite{\csname if#1\endcsname 00\else 01\fi}}
\setkeys{DMSS}{overwrite=false}
\setkeys{DMSS}{format*=}
\newcommand\MakeDeclareMathSetCommand[3]{%
   \expandafter\MakeDeclareShorthandCommandAux\csname math#2format\endcsname
   {#1}{#2}{#3}}
\def\MakeDeclareShorthandCommandAux#1#2#3#4{%
   \newcommand*#1{#4}%
   \newcommand*#2[2][]{%
      \DeclareMathSymbolShorthand{format=#1,prefix=#3,##1}{##2}%
   }}
\MakeDeclareMathSetCommand{\DeclareMathSet}{numbers}{\mathbb}
\makeatother

\MakeDeclareMathSetCommand{\DeclareMathNumbers}{numb}{\mathbb}
\MakeDeclareMathSetCommand{\DeclareMathGroup}{group}{\mathrm}
\DeclareMathNumbers{N} 
\DeclareMathNumbers{Z} 
\DeclareMathNumbers{Q} 
\DeclareMathNumbers{R} 
\DeclareMathNumbers{C} 
\DeclareMathGroup{GL} 
\DeclareMathGroup{SL} 
\DeclareMathGroup{Mat} 
\newcommand{\PSL}[0]{\mathbf{PSL}}

\newcommand{\U}[0]{\mathbf{U}}

\newcommand{\GL}[0]{\mathbf{GL}}
\newcommand{\SetH}[0]{\mathds{H}}

\newcommand{\tr}{\text{tr}}
\newcommand{\T}{\mathcal{T}}
\newcommand{\M}{\mathcal{M}}

\renewcommand{\phi}{\varphi}
\renewcommand{\epsilon}{\varepsilon}

\DeclareMathOperator{\Id}{Id}
\newcommand{\ra}{\rightarrow}
\title{Coordinates for the Universal Moduli Space of Holomorphic Vector Bundles}
\author{Jørgen Ellegaard Andersen and Niccolo Skovg{\aa}rd Poulsen}

\begin{document}

\maketitle 

\begin{abstract}
In this paper we provide two ways of constructing complex coordinates on the moduli space of pairs of a Riemann surface and a stable holomorphic vector bundle centred around any such pair. We compute the transformation between the coordinates to second order at the center of the coordinates. We conclude that they agree to second order, but not to third order at the center.

\end{abstract}

\section{Introduction}
\allowdisplaybreaks
Fix $g,n>1$ to be integers and let $d\in \{0, \ldots n-1\}.$ Let $\Sigma$ be a closed oriented surface of genus $g$.
Consider the universal moduli space $\M$ consisting of equivalence classes of pairs $(\phi: \Sigma \ra X,E)$ where $X$ is a Riemann Surface of genus $g$, $\phi : \Sigma \ra X$ is a diffeomorphism and $E$ is a semi-stable bundle over $X$ of rank $n$ and degree $d$. Let $\M^s$ be the open dense subset of $\M$ consisting of equivalence classes of such pairs $(\phi: \Sigma \ra X,E)$ with $E$ stable. The main objective of this paper is to provide coordinates in a neighbourhood of the equivalence class of any pair $(\phi: \Sigma \ra X,E)$ in $\mathcal M^s$. There is an obvious forgetful map 
$$\pi_\T :\M \ra \T$$
where $\T$ is the Teichm\"{u}ller space of $\Sigma$, whose fiber over $[\phi: \Sigma \ra X]\in \T$ is the moduli space of semi-stable bundles for that Riemann surface structure on $\Sigma$. Let $\pi_\T^s : \M^s \ra \T$ denote the restriction of $\pi_\T$ to $\M^s$, and we denote a point $[\phi: \Sigma \ra X]$ in $\T$ by $\sigma$.

We recall that locally around any $\sigma\in \T$ there are the Bers coordinates \cite{bers}. Further, for any point $[E]$ in some fiber $(\pi^{s}_\T)^{-1}(\sigma)$ we have the Zograf and Takhtadzhyan coordinates near $[E]$ along that fiber of $\pi_\T$ \cite{ZTVB}.

In order to describe our coordinates on $\M^s$ we recall the Narasimhan-Seshadri theorem. Let $\tilde\pi_1(\Sigma)$ be the universal central $\numbZ/n\numbZ$ extension of $\pi_1(\Sigma)$ and let $M$ be the moduli space of representations of $\tilde\pi_1(\Sigma)$ to $\U(n)$ such that the central generator goes to $e^{2\pi i d/n}\Id$. Let $M'$ be the subset of $M$ consisting of equivalence classes of irreducible representations. The Narasimhan-Seshadri theorem gives us a diffeomorphism
$$\Psi : \T\times M' \ra \M^s $$
which we use to induce a complex structure on $\T\times M'$ such that $\Psi$ is complex analytic. We will now represent a point in $\T$ by a representation
$$ \rho_0 : \tilde\pi_1(\Sigma) \ra \PSL(2)$$
and denote the corresponding point in Teichm\"{u}ller space by $X_{\rho_0}$. Here $\rho_0$ is really a representation of $\pi_1(\Sigma)$ pulled back to $\tilde\pi_1(\Sigma)$. A point in $M'$ will be represented by a representation
$$\rho_{E} : \tilde\pi_1(\Sigma) \ra \U(n)$$
which corresponds to the stable holomorphic bundle $E$ on $X_{\rho_0}$.

We build complex analytic coordinates around any such $(\rho_0,\rho_{E})\in\mathcal{T}\times M'$ by providing a complex analytic isomorphism from a small neighbourhood around 0 in the vector space $H^{0,1}(X_{\rho_0}, TX_{\rho_0})\oplus H^{0,1}(X_{\rho_0},\text{End}E)$ to a small open subset containing $(\rho_0,\rho_{E})$ in $ \mathcal{T}\times M'$. 

The coordinates are given by constructing a certain family
\begin{equation}
  \label{eq:phidomain}
\Phi^{\mu\oplus\nu}:\SetH\times \GL(n,\numbC)\to \SetH\times \GL(n,\numbC)
\end{equation}
of bundle maps of the trivial $\GL(n,\numbC)$-principal bundles over $\SetH$
indexed by pairs of sufficiently small elements $$\mu\oplus\nu\in H^{0,1}(X_{\rho_0}, TX_{\rho_0})\oplus H^{0,1}(X_{\rho_0},\text{End}E).$$
These bundle maps will uniquely determine  representations $(\rho^{\mu},\rho^{\mu\oplus\nu}_{E})\in\mathcal{T}\times M'$ such that
\begin{equation} \label{eqeq}
\rho^\mu(\gamma) \times\rho^{\mu\oplus\nu}_E(\gamma)=\Phi^{\mu\oplus\nu}\circ(\rho_0(\gamma)\times\rho_{E}(\gamma))\circ(\Phi^{\mu\oplus\nu})^{-1}
\end{equation}
for all $\gamma \in \tilde\pi_1(X)$ by the following theorem. Pick a base point $z_0 \in \SetH$ and let $p_{\GL(n,\numbC)}$ be the projection onto $\GL(n,\numbC)$ of the trivial bundle $\SetH\times \GL(n,\numbC)$.

\begin{thm}
\label{thm:uni coord}
For all sufficiently small $\mu\oplus\nu\in H^{0,1}(X_{\rho_0}, TX_{\rho_0})\oplus H^{0,1}(X_{\rho_0},\text{End}E)$ there exist a unique bundle map $\Phi^{\mu\oplus\nu}$ such that
\begin{enumerate}
\item $\Phi^{\mu\oplus\nu}$ solves 
\begin{equation}\label{eqM}
\bar \partial_\SetH\Phi^{\mu\oplus\nu} =\partial \Phi ^{\mu\oplus\nu} (\mu\oplus\nu),
\end{equation} 
where $\nu$ is considered a left-invariant vector field on $\GL(n,\numbC)$ at each point in $\SetH$.
\item The base map extends to the boundary of $\SetH$ and fixes $0,1$ and $\infty$.
\item The pair of representations $(\rho^\mu,\rho^{\mu\oplus\nu}_E)$ defined by equation (\ref{eqeq}) represents a point in $\mathcal{T}\times M'$.
\item $p_{\GL(n,\numbC)}(\Phi^{\mu\oplus\nu}(z_0,e))$ has determinant $1$ and is positive definite. 
\end{enumerate}
\end{thm}

From this theorem we easily derive our main theorem of this paper.

\begin{thm}
\label{thm:holafcord}
Mapping all sufficiently small pairs $$\mu\oplus\nu\in H^{0,1}(X_{\rho}, TX_{\rho})\oplus H^{0,1}(X_{\rho},\text{End}E)$$ 
to
$$ (\rho^\mu ,\rho^{\mu\oplus\nu}_E) \in \T\times M'$$
provides local analytic coordinates centered at $(\rho_0,\rho_{E})\in \T\times M'$.

\end{thm}

Our second coordinate construction provides fibered coordinates, which along $\T$ uses Bers' coordinates, \cite{bers}, and which uses Zograf and Takhtadzhyan's coordinates \cite{ZTVB} along the fibers. We refer to section \ref{sec:Fused} for the precise description of these fibered coordinates.

Finally, we compare the two sets of coordinates by computing the infinitesimal transformation of the coordinates up to second order at the center of both coordinates.
\begin{thm}
\label{thm:secondordvar}
  The fibered coordinates and the universal coordinates  agree to second order, but not the third order at the center of the coordinates.
\end{thm}

We refer to Theorem~\ref{thm:differthrid}, for the details of how the two set of coordinates differ at third-order.

With these new coordinates we get a new tool to analyse the metric and the curvature of the moduli space. Here we have taken the first step in understanding the curvature by calculating the second variation of the metric in local coordinates, at the center point. We intend to return to the full calculation of the curvature in these coordinates in a future publication.

\begin{rem}
If we perform our construction using elements of $\mathcal{H}^{0,1}(X,(\text{End}_0 E))$ where $(\text{End}_0 E)$ is the subspace of traceless endomorphisms, we get coordinates on the universal $\mathbf{SU}(n)$ moduli space in a completely similar way.
\end{rem}

\section{The Complex Structure on $\M^s$ from a Differential Geometric Perspective}

Recall that we endow the space $\mathcal{T}\times M$ with the structure of a complex manifold by using the Narasimhan-Seshadri theorem to provide us with the diffeomorphism
$$\Psi :  \T\times M' \ra  \M^s$$
and then declaring it to be complex analytic. There is the following alternative construction of this complex manifold structure.

Recall the general setting of \cite{JEANLGML} in the context of geometric quantization and the Hitchin connection, namely $\tilde{\mathcal{T}}$ is a general complex manifold and $(\tilde M,\omega)$ is a general symplectic manifold. In that paper a construction of a complex structure on $\tilde{\mathcal{T}}\times \tilde M$ is provided via the following proposition. But first we need the following definition.
\begin{defn}
  A family of K{\"a}hler structures on $(\tilde M,\omega)$ parametrized by $\tilde{\mathcal{T}}$ is called holomorphic if it satisfies:
  \begin{equation*}
    V'[J]=V[J]'\quad V''[J]=V[J]''
  \end{equation*}
for all vector fields $V$ on $\tilde{\mathcal{T}}$. Here the single prime on $V$ denotes projection on the $(1,0)$-part and the double prime on $V$ denotes projection on the $(0,1)$-part of the vector field $V$. Further $V[J]\in T_\sigma\otimes (\bar T_\sigma)^*\oplus\bar T_\sigma\otimes T_\sigma^*$ and we let $V[J]'$ denote the projection on the first, and $V[J]''$ the projection on the second factor.
\end{defn}
\begin{prop}[{\cite[Proposition 6.2]{JEANLGML}}]
  The family $J_\sigma$ of K{\"a}hler structures on $\tilde M$ is holomorphic, if and only if the almost complex structure $J$ given by \[ J(V\oplus X)= IV\oplus J_\sigma X,\quad \forall (\sigma,[\rho_E])\in \tilde{\mathcal{T}}\times \tilde M,  \forall (V,X)\in T_{\sigma,[{\rho_E}]}(\tilde{\mathcal{T}}\times \tilde M),\] is integrable.
\end{prop}

The family of complex structures on $M'$ considered in \cite{hitchin2} see also \cite{JEANLGML}, \cite{Andersen2} and \cite{Andersen1}, given by the Hodge star, $-\star_\sigma$, $\sigma\in \mathcal{T}$, fulfills the requirements of the proposition with respect to the Atiyah-Bott symplectic form $\omega$ on $M'$. We will denote the complex structure which $\mathcal{T}\times M'$ has by $J$. 

\begin{prop}
\label{prop:Csstrucmodspace}
We have that the map
$$\Psi: (\T\times M', J) \ra \M^s$$
is complex analytic, e.g. $J$ is in fact the complex analytic structure this space gets from the Narasimhan-Seshadri diffeomorphism $\Psi$.
\end{prop}
\begin{proof}

In order to understand the complex structure of $\mathcal{T}\times M'$ from the algebraic geometric perspective we want to construct holomorphic horizontal sections of $\mathcal{T}\times M'\to\mathcal{T}$.  We will use the universal property of the space of holomorphic bundles to show that the sections $\mathcal{T}\to\mathcal{T}\times\{\rho_E\}\subset \mathcal{T}\times M'$ are holomorphic for all $[\rho_E]$ in $M'$.

Our first objective is to construct a holomorphic family of vector bundles over Teichm{\"u}ller space, where each bundle corresponds to the same unitary representation of $\tilde\pi_1(\Sigma)$. We start from the universal curve $\mathcal{T}\times \Sigma$ and its universal cover $\mathcal{T}\times \tilde\Sigma$. Both of these spaces are complex analytic, and we get the universal curve $\mathcal{T}\times \Sigma$ as the quotient of $\mathcal{T}\times \tilde\Sigma$ by the holomorphic $\pi_1(\Sigma)$ action.

 This allows us to construct the vector bundles over $\mathcal{T}$ as the sheaf theoretic quotient of
\[\mathcal{T}\times \tilde\Sigma\times \numbC^n\]
by the $\tilde \pi_1(\Sigma)$ action, given by the $\pi_1(\Sigma)$-action on 
$\mathcal{T}\times \tilde\Sigma$, and the unitary action on $\numbC^n$ given by our fixed representation $\rho_E:\tilde\pi_1(\Sigma)\to \mathbf{U}(n)$ (see \cite{MS} for details on this construction, where we simply just compose representations with the natural quotient map from $\pi_1(\Sigma - \{p\})$ to $\tilde\pi_1(\Sigma)$ to match up the setting of this paper to a special case of the setting in \cite{MS}).  The action is of course holomorphic, and so the quotient (fiberwise invariant sections over $\mathcal{T}$)  is a family of Riemann surfaces with a holomorphic vector bundle over it of rank $n$ and degree $d$. The universal property of $\mathcal{M}^s$ implies that this family therefore induced a holomorphic section
\begin{align*}
 \iota_{\rho_E}:\mathcal{T}&\to  \mathcal{M}^s
\end{align*}
by the universality of the moduli space $\mathcal{M}^s$. This shows that the horizontal sections are holomorphic submanifolds, and so the tangent space must split at every point as $I\oplus J_\sigma$. Here $J_\sigma$ must be $-\star_\sigma$ since it comes from the structure of the fibers. 
\end{proof}
The conclusion is, that the algebraic complex structure on the moduli space of pairs of a Riemann surface and a holomorphic vector bundle over it and the complex structure from \cite{JEANLGML} on $\T\times M'$ are the same.

\section{Coordinates for the Universal Moduli Space of Holomorphic Vector Bundles}
\label{sec:UniCoord}

In this section we prove Theorem \ref{thm:uni coord}.

We will need the composition of the map $ \Phi^{\mu\oplus\nu}$ with the projection on each of the two factors, which we denote as follows:
\begin{align*}
  \Phi_1^{\mu\oplus\nu}:\SetH\times \GL(n,\numbC)\to \SetH,\\\Phi_2^{\mu\oplus\nu}:\SetH\times \GL(n,\numbC)\to \GL(n,\numbC).
\end{align*}
In fact $\Phi_1^{\mu\oplus\nu}$ is the projection onto $\SetH$ followed by the induced map on the base by \eqref{eq:basemap} below.

The equation (\ref{eqM}) is equivalent to the following two equations on $\Phi_i^{\mu\oplus\nu}$:
\begin{align}
  \label{eq:ucoord1}
  \bar\partial_\SetH\Phi_1^{\mu\oplus\nu}(z,g)&=\mu\partial_\SetH\Phi_1^{\mu\oplus\nu}(z,g)\\
 \label{eq:ucoord2} 
\bar\partial_\SetH\Phi_2^{\mu\oplus\nu}(z,g)&=\mu\partial_\SetH\Phi_2^{\mu\oplus\nu}(z,g)+\partial_{\GL(n,\numbC)}\Phi_2^{\mu\oplus\nu}(z,g)\nu.
\end{align}
 since $\partial_{\GL(n,\numbC)}\Phi_1^{\mu\oplus\nu}(z,g)=0$. With this simplification the first equation is exactly Bers's equation for
 \begin{equation}
   \label{eq:basemap}
   \Phi_1^{\mu}(z) = \Phi_1^{\mu\oplus\nu}(z,g),
 \end{equation}
and so we can solve it using the techniques in \cite{bers}, and we obtain a Riemann surface $X_{\rho_\mu}$ corresponding to a representation $\rho_\mu$.

The second equation \eqref{eq:ucoord2} we solve in two steps. First, we identify $\nu$ with an endomorphism valued $1$-form using the standard identification of left invariant vector fields and elements of the Lie algebra. To solve the equation we consider the anti-holomorphic solution of the equation
\[ \bar\partial_\SetH\Phi_-^{\nu}(z,e)=\partial_{\GL(n,\numbC)}\Phi_-^{\nu}(z,g)(\nu)\vert_{g=e}  =\Phi_-^{\nu}(z,e)\cdot\nu\] 
and extend it equivariantly to the rest of $\SetH\times \GL(n,\numbC)$. We observe that $\partial_\SetH \Phi_-^{\nu}=0$ since it is anti-holomorphic. And so it follows, by adding $0$ to the defining equation of $\Phi_-^{\nu}$ that: 
\[ \bar\partial_\SetH\Phi_-^{\nu}(z,g)=\partial_{\GL(n,\numbC)}\Phi_-^{\nu}(z,g)(\nu)+\mu\partial_\SetH\Phi_-^{\nu}(z,g).\]
The vector bundle on $X_{\rho_\mu}$ corresponding to the representation \[\chi^{\nu}(\gamma)=\Phi_-^{\nu}(\rho_0(\gamma) z,e)\rho_{E}^{0\oplus0}(\gamma)(\Phi_-^{\nu}(z,e))^{-1}\] is stable, if $\mu\oplus\nu$ is small enough. This means, we can find a holomorphic gauge transformation on the universal cover of $X_{\rho_\mu}$, $\Phi_+^{\mu\oplus\nu}:\SetH\to\GL(n,\numbC)$,  such that \begin{equation}\label{eq:rhomunu}\rho_{E}^{{\mu\oplus\nu}}(\gamma)=\Phi_+^{\mu\oplus\nu}(\rho_\mu(\gamma) z)\chi^{\mu\oplus\nu}(\gamma)(\Phi_+^{\mu\oplus\nu}(z))^{-1}
\end{equation}
is an admissible $\mathbf{U}(n)$-representation and independent of $z$ by the Narasimhan-Seshadri theorem \cite{NandS}. Now we use the basemap to define $\tilde \Phi_+^{\mu\oplus\nu}=\Phi_+^{\mu\oplus\nu}\circ \Phi_1^{\mu}$.



The following computation shows that the map $\tilde\Phi_+^{\mu\oplus\nu}$ is in the kernel of $\bar\partial_\SetH-\mu\partial_\SetH$:

\begin{align*}
  (\bar\partial_\SetH-\mu\partial_\SetH) \tilde\Phi_+^{\mu\oplus\nu}&= (\bar\partial_{\SetH} \Phi_+^{\mu\oplus\nu})\circ\Phi_1^{\mu} \bar\partial_{\SetH}\bar\Phi_1^{\mu}+(\partial_{\SetH} \Phi_+^{\mu\oplus\nu})\circ\Phi_1^{\mu} \bar\partial_{\SetH}\Phi_1^{\mu} \\ &\quad -\mu(\bar\partial_{\SetH} \Phi_+^{\mu\oplus\nu})\circ\Phi_1^{\mu} \partial_{\SetH}\bar\Phi_1^{\mu}-\mu(\partial_{\SetH} \Phi_+^{\mu\oplus\nu})\circ\Phi_1^{\mu} \partial_{\SetH}\Phi_1^{\mu}
\end{align*}
We then use the differential equation $\bar\partial\Phi_1^{\mu}=\mu \partial\Phi_1^{\mu}$ and that $\bar\partial_\SetH\Phi^{\mu\oplus\nu}_+=0$ to get that
\begin{align*}
  (\bar\partial_\SetH-\mu\partial_\SetH) \tilde\Phi_+^{\mu\oplus\nu}=\partial_{\SetH} \Phi_+^{\mu\oplus\nu}\circ\Phi_1^{\mu} \mu\partial_{\SetH}\Phi_1^{\mu} -\mu\partial_{\SetH} \Phi_+^{\mu\oplus\nu}\circ\Phi_1^{\mu} \partial_{\SetH}\Phi_1^{\mu}=0.
\end{align*}
Define $\Phi_2^{\mu\oplus\nu}(z,g)=\tilde\Phi_+^{\mu\oplus\nu}(z,g)\Phi_-^{\nu}(z,g)$. We see that $\Phi_2^{\mu\oplus\nu}$ fulfills equation \eqref{eq:ucoord2} by the following calculation
\begin{align*}
  \bar\partial_\SetH \Phi_2^{\mu\oplus\nu}&=(\bar\partial_\SetH \tilde{\Phi}_+^{\mu\oplus\nu})(\Phi_-^{\nu})+( \tilde{\Phi}_+^{\mu\oplus\nu})(\bar\partial_\SetH \Phi_-^{\nu})\\&=(\bar\partial_\SetH \tilde{\Phi}_+^{\mu\oplus\nu})(\Phi_-^{\nu})+( \tilde{\Phi}_+^{\mu\oplus\nu})(\partial_{\GL(n,\numbC)} \Phi_-^{\nu}\nu)
\end{align*}
since $\tilde\Phi_+^{\mu\oplus\nu}\in\ker(\bar\partial_\SetH-\mu\partial_\SetH)$ we get that
\begin{equation*}
  \bar\partial_\SetH \Phi_2^{\mu\oplus\nu}=(\mu\partial_\SetH \tilde{\Phi}_+^{\mu\oplus\nu})(\Phi_-^{\nu})+( \tilde{\Phi}_+^{\mu\oplus\nu})(\partial_{\GL(n,\numbC)} \Phi_-^{\nu}\nu).
\end{equation*}
 To finish the calculation we use that $\Phi_+$ and $\Phi_1$ are independent of the $\GL(n,\numbC)$ factor, and therefore so is $\tilde\Phi^{\mu\oplus\nu}_+$. Also $\Phi_-^{\mu\oplus\nu}$ is antiholomorphic so we have that
\begin{align*}
  \bar\partial_\SetH \Phi_2^{\mu\oplus\nu}&=\mu\partial_\SetH (\tilde{\Phi}_+^{\mu\oplus\nu}\Phi_-^{\nu})+\partial_{\GL(n,\numbC)} (\tilde{\Phi}_+^{\mu\oplus\nu}\Phi_-^{\nu})\nu\\ &=\mu\partial_\SetH \Phi_2^{\mu\oplus\nu}+(\partial_{\GL(n,\numbC)}\Phi_2^{\mu\oplus\nu})\nu
\end{align*}
To show that we still get an admissible representation, we use that \eqref{eq:rhomunu} is independent of which $z$ we choose. This lets us conclude that
\begin{align*}
  \rho_E^{{\mu\oplus\nu}}(\gamma)=&\Phi_+^{\mu\oplus\nu}(\rho_\mu(\gamma) \Phi_1^{\mu}(z))\chi^{\mu\oplus\nu}(\gamma)(\Phi_+^{\mu\oplus\nu}(\Phi_1^{\mu}(z)))^{-1}\\ =& \Phi_+^{\mu\oplus\nu}(\Phi_1^{\mu}(\rho_0(\gamma)(\Phi_1^{\mu})^{-1}(\Phi_1^{\mu}( z))))\chi^{\mu\oplus\nu}(\gamma)(\Phi_+^{\mu\oplus\nu}(\Phi_1^{\mu}(z)))^{-1}\\ =& \tilde\Phi_+^{\mu\oplus\nu}(\rho_0(\gamma) z)\chi^{\mu\oplus\nu}(\gamma)(\tilde\Phi_+^{\mu\oplus\nu}(z))^{-1},
\end{align*}
and so
\[ \rho_E^{{\mu\oplus\nu}}(\gamma)=\Phi_2^{\mu\oplus\nu}(\rho_0(\gamma)z,g)\rho_E^{0\oplus 0}(\gamma)(\Phi_2^{\mu\oplus\nu}(z,g))^{-1}\] 
is an admissible $\mathbf{U}(n)$-representation. Finally, the requirement that $\Phi_2^{\mu\oplus\nu}(z_0,e)$ is a positive definite matrix of determinant $1$ fixes all remaining indeterminacy as in \cite{ZTVB}.

\subsection{The Tangent Map from Kodaira-Spencer Theory}
\label{sec:kodbtz}
We will now analyse the tangential map of our coordinates. The only problematic part is what happens in the tangent directions parallel to the fibers. We can calculate the Kodaira-Spencer map of the family of representations $\rho_E^{\mu\oplus\nu+t\tilde\mu\oplus\tilde\nu}, t\in\numbC$. However, to ease the computation we first prove the following lemma.

\begin{lemma}
  \label{lem:easecomp}
We let $X_{\rho_0}$ be a Riemann surface and $\rho_0$ the corresponding representation of $\pi_1(X_{\rho_0})$. For a family of representations of $R_t:\tilde \pi_1(X_{\rho_0})\to \mathbf{U}(n)$, where 
  $$R_t(\gamma)=\Upsilon(t,\rho_0(\gamma) z)\rho_E(\gamma)\Upsilon(t,z)^{-1}$$
  with both $\rho_0$ and $\rho_E$ independent of $t$ and $ \Upsilon$ any smooth map $$\Upsilon:\numbC\times\SetH\to\GL(n,\numbC),$$ we have that the Kodaira-Spencer class's harmonic representative of the family $R_t$ at $t=0$ is:
$$ P^{0,1}_{\rho_0,E}\left(\text{Ad}\Upsilon(0,z)\left(\left.\frac{d}{dt}\right\vert_{t=0} \Upsilon(t,z)^{-1}\bar\partial_{\SetH} \Upsilon(t,z)\right)\right)\in H^{0,1}(X_0,\text{End}E_{R_0}).$$
Here $P^{0,1}_{\rho_0,E}$ denotes the projection on the harmonic forms on $X_{\rho_0}$ with values in $\textnormal{End}E_{R_0}$.
\end{lemma}
\begin{proof}
  To compute the Kodaira-Spencer map we first consider $\left.\frac{d}{dt}\right\vert_{t=0} R_t$ and note, this is an element of $H^1(X,\text{End}(E))$. However, this cohomology group is isomorphic to $H^{0,1}(X,\text{End}E)$. The isomorphism is constructed by finding a {\v C}ech chain  with values in the sheaf $\Omega^1(\text{End}E)$, say $\varphi_i$,  such that 
  $$\delta^*(\phi)_{ij}=\phi_i-\phi_j =\left.\frac{d}{dt}\right\vert_{t=0} R_t(\gamma_{ij}),$$ 
  for open sets $U_i\cap U_j\neq\emptyset$ which are related by the transformation $\gamma_{ij}\in\tilde\pi_1(\Sigma)$ on the universal cover. Once $\phi_i$ has been found, $P_{\rho_0,E}^{0,1}(\bar\partial_\SetH \phi_i)$ will give a harmonic representative of the Kodaira-Spencer class.

We can now calculate that
\begin{align*}
  \left.\frac{d}{dt}\right\vert_{t=0}R_t(\gamma_{ij})&=\left.\frac{d}{dt}\right\vert_{t=0}\Upsilon(t,\rho_0(\gamma) z)\rho_E(\gamma)\Upsilon(t,z)^{-1}\\ &=\left.\frac{d}{dt}\right\vert_{t=0}\Upsilon(t,\rho_0(\gamma) z)\rho_E(\gamma)\Upsilon(0,z)^{-1}\\ &\quad+\Upsilon(0,\rho_0(\gamma) z)\rho_E(\gamma)\left.\frac{d}{dt}\right\vert_{t=0}\Upsilon(t,z)^{-1}
\\ &= \left.\frac{d}{dt}\right\vert_{t=0}(\Upsilon(t,\rho_\SetH(\gamma_{ij})z)\Upsilon(0,\rho_0(\gamma_{ij})z)^{-1})R_0(\gamma_{ij})\\ &\quad-R_0(\gamma_{ij})\left.\frac{d}{dt}\right\vert_{t=0}(\Upsilon(t,z)\Upsilon(0,z)^{-1})\\ &=R_0(\gamma_{ij})\delta(\left.\frac{d}{dt}\right\vert_{t=0}(\Upsilon(t,z)\Upsilon(0,z)^{-1}))_{ij}
\end{align*}
The Kodaira-Spencer class is then:
\begin{align*}
  \bar\partial_{\SetH} \left.\frac{d}{dt}\right\vert_{t=0}(\Upsilon(t,z)\Upsilon(0,z)^{-1}) =\text{Ad}\Upsilon(0,z)\left(\left.\frac{d}{dt}\right\vert_{t=0} \Upsilon(t,z)^{-1}\bar\partial_{\SetH} \Upsilon(t,z)\right).
\end{align*}
We compose with the harmonic projection to get the harmonic representative.
\end{proof}
We have the following proposition.
\begin{prop}\label{prop:KSuniv}
  The Kodaira-Spencer   map of $\rho_E^{\mu\oplus\nu+t\tilde\mu\oplus\tilde\nu}, t\in\numbC$ at $\mu\oplus\nu\in H^{0,1}(X,TX)\oplus H^{0,1}(X,\textnormal{End}E)$,
  $$KS_{\mu\oplus\nu}: H^{0,1}(X,TX)\oplus H^{0,1}(X,\textnormal{End}E)\to H^{0,1}(X_{\rho_\mu},TX_{\rho_\mu})\oplus H^{0,1}(X_{\rho_\mu},\textnormal{End}E_{\rho_E^{\mu\oplus\nu}})$$ 
  is given by
  \begin{multline*}
    KS_{\mu\oplus\nu}(\tilde\mu\oplus\tilde\nu)=P_{\mu}\tilde\mu^{\mu}\oplus P_{\mu\oplus\nu}^{0,1}\left((\Phi_1^{\mu})_*^{-1}\left(\text{Ad}\Phi_2^{\mu\oplus\nu}\left(\tilde\mu(\Phi_2^{\mu\oplus\nu})^{-1}\partial\Phi_2^{\mu\oplus\nu}+\tilde\nu\right)\right)\right).
  \end{multline*}
Here $\tilde\mu^{\mu}=(\frac{\tilde\mu}{1-\vert{\mu}\vert^2}\frac{\partial \Phi_1^{\mu}}{\overline{\partial \Phi_1^{\mu}}})\circ (\Phi_1^{\mu})^{-1}$ and $P_{\mu}$ and $P_{\mu\oplus\nu}^{0,1}$ are the $L^2$-projections on the harmonic forms $H^{0,1}(X_{\rho_\mu},TX_{\rho_\mu})$ respectively $H^{0,1}(X_{\rho_\mu},\textnormal{End}E_{\rho_E^{\mu\oplus\nu }})$.
\end{prop}
\begin{proof}
By using that the defining equation (\ref{eqeq}) for $\rho_E^{\mu\oplus\nu+t\tilde\mu\oplus\tilde\nu}$ is independent of $z$, we get that
\begin{align*}
  \rho_E^{\mu\oplus\nu+t\tilde\mu\oplus\tilde\nu}& =\Phi_2^{\mu\oplus\nu+t\tilde\mu\oplus\tilde\nu}( \rho_{0}(\gamma) z),e)\rho_E^{0\oplus0}(\gamma) \Phi_2^{\mu\oplus\nu}(z,e)^{-1}\\ &=\Phi_2^{\mu\oplus\nu+t\tilde\mu\oplus\tilde\nu}((\Phi_1^{\mu})^{-1}( \rho_{\mu}(\gamma) z),e)\rho_E^{0\oplus0}(\gamma)\\ &\qquad \cdot  \Phi_2^{\mu\oplus\nu}((\Phi_1^{\mu})^{-1}(z),e)^{-1}.
\end{align*}
And so to find the Kodaira-Spencer class, by Lemma \ref{lem:easecomp} we only need to calculate:
\begin{align}
\nonumber \text{Ad}\Phi_2^{\mu\oplus\nu}&\circ(\Phi_1^{\mu})^{-1}\frac{d}{dt}\vert_{t=0}(\Phi_2^{\mu\oplus\nu+t\tilde\mu\oplus\tilde\nu}\circ(\Phi_1^{\mu})^{-1})^{-1}\bar\partial(\Phi_2^{\mu\oplus\nu+t\tilde\mu\oplus\tilde\nu}\circ(\Phi_1^{\mu})^{-1})\\ =&\text{Ad}\Phi_2^{\mu\oplus\nu}\circ(\Phi_1^{\mu})^{-1} \nonumber\\ &\cdot\frac{d}{dt}\vert_{t=0}((t\tilde\mu(\Phi_2^{\mu\oplus\nu+t\tilde\mu\oplus\tilde\nu})^{-1}\partial\Phi_2^{\mu\oplus\nu+t\tilde\mu\oplus\tilde\nu})\circ\Phi_1^{\mu})^{-1}\bar\partial(\bar\Phi_1^{\mu})^{-1}\nonumber\\ &+\text{Ad}\Phi_2^{\mu\oplus\nu}\circ(\Phi_1^{\mu})^{-1}\frac{d}{dt}\vert_{t=0}(\nu+t\tilde\nu)\circ(\Phi_1^{\mu})^{-1}\bar\partial(\bar\Phi_1^{\mu})^{-1}\nonumber\\ =&(\Phi_1^{\mu})_*^{-1}\left(\text{Ad}\Phi_2^{\mu\oplus\nu}\left(\tilde\mu(\Phi_2^{\mu\oplus\nu})^{-1}\partial\Phi_2^{\mu\oplus\nu}+\tilde\nu\right)\right).\label{KSunivcoord}
\end{align}
Now to get the Kodaira-Spencer map we project on the harmonic $(0,1)$-forms and remark that in the Teichm{\"u}ller directions we can apply the usual arguments from the classical case of Bers's coordinates.
\end{proof}
We see the map is injective and complex linear in both $\tilde\mu$ and $\tilde\nu$. Since we know $\mathcal{T}\times M'$ is a manifold the Implicit Function Theorem now implies that the coordinates we constructed are in fact holomorphic coordinates in a small neighbourhood. This completes the proof of Theorem \ref{thm:holafcord}.

\section{The Fibered Coordinates}
\label{sec:Fused}

In this section we will fuse Zograf and Takhtadzhyan's coordinates with Bers's coordinates in a kind of fibered manner in order also to produce coordinates on $\mathcal{T}\times M'$, which are complex analytic with respect to $J$.

Since we trough any stable bundle have a copy of $\mathcal{T}$ embedded as a complex submanifold,  we can construct fibered coordinates, once we identify the tangent spaces in the fiber direction locally along these copies of $\mathcal{T}$. We identify them by the maps 
\[ H^{0,1}(X_{\rho_0},\text{End}E_{\rho_E^{0}})\ni \nu\to\nu^\mu=P^{0,1}_{\mu} ((\Phi_1^\mu)_*^{-1}(\nu))\in H^{0,1}(X_{\rho_\mu},\text{End}E_{\rho_E^{0^\mu}}).\] 
This identification gives us coordinates taking $(\mu,\nu)$ to \[(\rho_\mu,\rho_E^{{\nu^\mu}})=\left(\Phi_1^\mu\circ \rho_0(\gamma)\circ (\Phi_1^\mu)^{-1},f^{\nu^\mu}(\rho_\mu(\gamma) z)\rho_{E}^{0^0}(\gamma)(f^{\nu^\mu}(z))^{-1}\right).\] These are complex coordinates, since $\nu^\mu$ are local holomorphic sections of the tangent bundle.

Before we calculate the Kodaira-Spencer maps for these coordinate curves, we will need to understand the derivatives of $(\Phi_1^\mu)^{-1}$.
\begin{lemma}
\label{lem:idenditeise}
We have the following two identities for $(\Phi_1^\mu)^{-1}:\SetH\to\SetH$
\begin{align}\label{eq:andet led}\bar\partial(\Phi_1^{\mu})^{-1}&= -\mu\circ(\Phi_1^{\mu})^{-1}\bar\partial\overline{(\Phi_1^{\mu})^{-1}}.
\\
  \label{eq:overlines}
\bar\partial\overline{(\Phi_1^{\mu})^{-1}}&=\left(\frac{1}{1-\vert\mu\vert^2}\frac{1}{\overline{\partial \Phi_1^\mu}}\right)\circ (\Phi_1^\mu)^{-1}.\end{align}
\end{lemma}
\begin{proof}
We consider the identity $\Phi_1^\mu\circ (\Phi_1^\mu)^{-1}(z)=z$. And we use the differential equation for $\Phi_1^\mu$ which is $$\bar\partial\Phi^\mu_1=\mu\partial\Phi^\mu_1$$ to calculate:
\begin{align*}
 0= \bar\partial( \Phi_1^{\mu}\circ(\Phi_1^{\mu})^{-1})&= (\bar\partial\Phi_1^{\mu})\circ(\Phi_1^{\mu})^{-1}\bar\partial\overline{(\Phi_1^{\mu})^{-1}}+(\partial\Phi_1^{\mu})\circ(\Phi_1^{\mu})^{-1}\bar\partial(\Phi_1^{\mu})^{-1}\\ &= (\mu\partial\Phi_1^{\mu})\circ(\Phi_1^{\mu})^{-1}\bar\partial\overline{(\Phi_1^{\mu})^{-1}}+(\partial\Phi_1^{\mu})\circ(\Phi_1^{\mu})^{-1}\bar\partial(\Phi_1^{\mu})^{-1}.
\end{align*}
Now $\partial\Phi_1^{\mu}\neq0$ for $\mu$ small, since $\Phi_1$ is a continuous perturbation of the identity map $\Id:\SetH\to\SetH$.  We then calculate that \begin{equation*} -\mu\circ(\Phi_1^{\mu})^{-1}\bar\partial\overline{(\Phi_1^{\mu})^{-1}}=\bar\partial(\Phi_1^{\mu})^{-1},
\end{equation*}
which is (\ref{eq:andet led}). We can use (\ref{eq:andet led}) to describe $\bar\partial\overline{(\Phi_1^{\mu})^{-1}}$ by differentiating $\Phi_1^{\mu}\circ(\Phi_1^{\mu})^{-1}(z)=z$:
\begin{align*}
  1&=\partial( \Phi_1^{\mu}\circ(\Phi_1^{\mu})^{-1})= (\bar\partial\Phi_1^{\mu})\circ(\Phi_1^{\mu})^{-1}\partial\overline{(\Phi_1^{\mu})^{-1}}+(\partial\Phi_1^{\mu})\circ(\Phi_1^{\mu})^{-1}\partial(\Phi_1^{\mu})^{-1}\\ &=
-\mu(\partial\Phi_1^{\mu})\circ(\Phi_1^{\mu})^{-1}\bar\mu\bar\partial\overline{(\Phi_1^{\mu})^{-1}}+(\partial\Phi_1^{\mu})\circ(\Phi_1^{\mu})^{-1}\partial(\Phi_1^{\mu})^{-1}\\ &=((1-\vert\mu\vert^2)(\partial\Phi_1^{\mu}))\circ(\Phi_1^{\mu})^{-1}\partial(\Phi_1^{\mu}))^{-1},
\end{align*}
and so conjugating and isolating $\bar\partial\overline{(\Phi_1^{\mu})^{-1}}$ we find:
\begin{equation*}
\bar\partial(\overline{\Phi_1^{\mu}})^{-1}=\left(\frac{1}{1-\vert\mu\vert^2}\frac{1}{\overline{\partial \Phi_1^\mu}}\right)\circ (\Phi_1^\mu)^{-1}\end{equation*} 
which proves (\ref{eq:overlines}).
\end{proof}
Let $\kappa_\mu$ be an $(n,m)$-tensor with values in the holomorphic bundle $E_{\rho_E^{0^\mu}}$ on the Riemann surface $X_{\rho_\mu}$ i.e.  
$$\kappa_\mu \in C^\infty(X_{\rho_\mu}, T^{-n}X_{\rho_\mu}\otimes \overline{T}^{-m}X_{\rho_\mu}\otimes \text{End}{E_{\rho_E^{0^\mu}}}).$$ 
Then we have that $(\Phi_1^\mu)_*(\kappa_\mu)=(\kappa_\mu\circ\Phi_1^\mu)(\partial \Phi_1^\mu)^n(\overline{\partial \Phi_1^\mu})^m$ and so $$(\Phi_1^\mu)_*^{-1}(\kappa_0)=(\kappa_0\circ(\Phi_1^\mu)^{-1})(\partial \Phi_1^\mu)^{-n}(\overline{\partial \Phi_1^\mu})^{-m}.$$ We have the families of unbounded operators
\begin{align*}
  \bar\partial_{\mu,E_{\rho_E^{0^\mu}}}: L^2(X_{\rho_\mu},\text{End}{E_{\rho_E^{0^\mu}}})&\to L^2(X_{\rho_\mu},T^{0,1}\otimes\text{End}{E_{\rho_E^{0^\mu}}}), \\ \bar\partial_{\mu,E_{\rho_E^{0^\mu}}}^*:L^2(X_{\rho_\mu},T^{0,1}\otimes\text{End}{E_{\rho_E^{0^\mu}}})&\to L^2(X_{\rho_\mu},\text{End}{E_{\rho_E^{0^\mu}}})\\ \Delta_{\mu,E_{\rho_E^{0^\mu}}}=\bar\partial^*\bar\partial:L^2(X_{\rho_\mu},\text{End}{E_{\rho_E^{0^\mu}}})&\to L^2(X_{\rho_\mu},\text{End}{E_{\rho_E^{0^\mu}}}),
\end{align*}
and the finite range operator
\begin{align*}P^{0,1}_{\mu,E_{\rho_E^{0^\mu}}}&: L^2(X_{\rho_\mu},T^{0,1}\otimes\text{End}{E_{\rho_E^{0^\mu}}})\to L^2(X_{\rho_\mu},T^{0,1}\otimes\text{End}{E_{\rho_E^{0^\mu}}})
\end{align*}
given by
$$
P^{0,1}_{\mu,E_{\rho_E^{0^\mu}}}=I-\bar\partial_{\mu,E_{\rho_E^{0^\mu}}}\Delta_{0,\mu,E_{\rho_E^{0^\mu}}}^{-1}\bar\partial_{\mu,E_{\rho_E^{0^\mu}}}^* ,
$$
where $\Delta_{0,\mu,E_{\rho_E^{0^\mu}}}$ is the restriction of $\Delta_{\mu,E_{\rho_E^{0^\mu}}}$ to the orthogonal complement of the subspace consisting of constant functions tensor the identity, and $P^{0,1}$ is the projection on the harmonic $(0,1)$-forms. 
We will also need the following results of Takhtajan and Zograf.
\begin{lemma}[{\cite{ZTpuncRie}}]\label{lem:projderiv}
  We have the following variational formulae for the derivative at $X_{\rho_0}$
  \begin{align*}
    \frac{d}{dt}\vert_{t=0} (\Phi_1^{t\tilde\mu})_*\bar\partial_{t\tilde\mu,E_{\rho_E^{0^\mu}}} (\Phi_1^{t\tilde\mu})_*^{-1}&=-\tilde\mu\partial_{0,E} \qquad   \frac{d}{d\bar t}\vert_{t=0} (\Phi_1^{t\tilde\mu})_*\bar\partial_{t\tilde\mu,E_{\rho_E^{0^\mu}}} (\Phi_1^{t\tilde\mu})_*^{-1}=0\\ \frac{d}{dt}\vert_{t=0} (\Phi_1^{t\tilde\mu})_*\bar\partial_{t\tilde\mu,E_{\rho_E^{0^\mu}}}^* (\Phi_1^{t\tilde\mu})_*^{-1}&=0\qquad  \frac{d}{dt}\vert_{t=0} (\Phi_1^{t\tilde\mu})_*\bar\partial_{t\tilde\mu,E_{\rho_E^{0^\mu}}}^* (\Phi_1^{t\tilde\mu})_*^{-1}=-\partial_{0,E}^*\overline{\tilde\mu}.
  \end{align*}
We further have at $(X_{\rho_\mu},E_{\rho_E^{0^\mu}})$, that
  \begin{align*}
    \frac{d}{dt}\vert_{t=0} (\Phi_1^{\mu+t\tilde\mu})_*&P_{\mu+t\tilde\mu,E_{\rho_E^{0^{\mu+t\tilde\mu}}}}^{0,1} (\Phi_1^{\mu+t\tilde\mu})_*^{-1}  \\ &=(\Phi_1^{\mu})_* P_{\mu,E_{\rho_E^{0^{\mu}}}}^{0,1}(\Phi_1^{\mu})_*^{-1}  \frac{d}{dt}\vert_{t=0} (\Phi_1^{\mu+t\tilde\mu})_*\bar\partial_{\mu+t\tilde\mu,E_{\rho_E^{0^{\mu+t\tilde\mu}}}} (\Phi_1^{\mu+t\tilde\mu})_*^{-1} \\ & \qquad  (\Phi_1^{\mu})_*\Delta^{-1}_{0,\mu,E_{\rho_E^{0^{\mu}}}}\bar\partial_{\mu,E_{\rho_E^{0^{\mu}}}}^*(\Phi_1^{\mu})_*^{-1}\\ & + (\Phi_1^{\mu})_*\bar\partial_{\mu,E_{\rho_E^{0^{\mu}}}} \Delta^{-1}_{0,\mu,E_{\rho_E^{0^{\mu}}}}(\Phi_1^{\mu})_*^{-1} \vert_{t=0} (\Phi_1^{\mu+t\tilde\mu})_*\bar\partial^*_{\mu+t\tilde\mu,E_{\rho_E^{0^{\mu+t\tilde\mu}}}} (\Phi_1^{\mu+t\tilde\mu})_*^{-1}\\ & \qquad\frac{d}{dt}(\Phi_1^{\mu})_* P_{\mu,E_{\rho_E^{0^{\mu+t\tilde\mu}}}}^{0,1}(\Phi_1^{\mu})_*^{-1}. \end{align*}
\end{lemma}
\begin{proof}
The first identities are proven in \cite[Equation (2.6)]{ZTpuncRie} (without the $\text{End}E$ factor, which makes no difference), the last statement is seen straightforwardly as follows
\begin{multline*}
  \frac{d}{dt}\vert_{t=0} (\Phi_1^{\mu+t\tilde\mu})_* P_{\mu+t\tilde\mu,E_{\rho_E^{0^{\mu+t\tilde\mu}}}}^{0,1} (\Phi_1^{\mu+t\tilde\mu})_*^{-1}= \frac{d}{dt}\vert_{t=0}(\Phi_1^{\mu+t\tilde\mu})_* \bar\partial_{\mu+t\tilde\mu,E_{\rho_E^{0^{\mu+t\tilde\mu}}}} (\Phi_1^{\mu+t\tilde\mu})_*^{-1} \\(\Phi_1^{\mu+t\tilde\mu})_*\Delta_{0,\mu+t\tilde\mu,E_{\rho_E^{0^{\mu+t\tilde\mu}}}}^{-1}(\Phi_1^{\mu+t\tilde\mu})_*^{-1} (\Phi_1^{\mu+t\tilde\mu})_*\bar\partial_{\mu+t\tilde\mu,E_{\rho_E^{0^{\mu+t\tilde\mu}}}}^* (\Phi_1^{\mu+t\tilde\mu})_*^{-1}.
\end{multline*}
We can then use the following identities 
\begin{align*}
   \frac{d}{dt}\vert_{t=0} (\Phi_1^{\mu+t\tilde\mu})_*\Delta_{0,\mu+t\tilde\mu,E_{\rho_E^{0^{\mu+t\tilde\mu}}}}^{-1}(\Phi_1^{\mu+t\tilde\mu})_*^{-1}&=-(\Phi_1^{\mu})_*\Delta_{0,\mu,E_{\rho_E^{0^{\mu}}}}^{-1}(\Phi_1^{\mu})_*^{-1}\\ \frac{d}{dt}\vert_{t=0}(\Phi_1^{\mu+t\tilde\mu})_*&\Delta_{0,\mu+t\tilde\mu,E_{\rho_E^{0^{\mu+t\tilde\mu}}}}(\Phi_1^{\mu+t\tilde\mu})_*^{-1}(\Phi_1^{\mu})_*\Delta_{0,\mu,E_{\rho_E^{0^{\mu}}}}^{-1}(\Phi_1^{\mu})_*^{-1},\\ 
\frac{d}{dt}\vert_{t=0}(\Phi_1^{\mu+t\tilde\mu})_*\Delta_{0,\mu+t\tilde\mu,E_{\rho_E^{0^{\mu+t\tilde\mu}}}}(\Phi_1^{\mu+t\tilde\mu})_*^{-1}& \\=\frac{d}{dt}\vert_{t=0}(\Phi_1^{\mu+t\tilde\mu})_*&\bar\partial_{\mu+t\tilde\mu,E_{\rho_E^{0^{\mu+t\tilde\mu}}}}(\Phi_1^{\mu+t\tilde\mu})_*^{-1}(\Phi_1^{\mu})_*\bar\partial_{\mu,E_{\rho_E^{0^{\mu}}}}^* (\Phi_1^{\mu})_*^{-1}\\ +\frac{d}{dt}\vert_{t=0}(\Phi_1^{\mu})_*&\bar\partial_{\mu,E_{\rho_E^{0^{\mu}}}}(\Phi_1^{\mu})_*^{-1} (\Phi_1^{\mu+t\tilde\mu})_*\bar\partial_{\mu+t\tilde\mu,E_{\rho_E^{0^{\mu+t\tilde\mu}}}}^* (\Phi_1^{\mu+t\tilde\mu})_*^{-1}.
\end{align*}
Now, putting this together and using that $P^{0,1}_{\mu,E_{\rho_E^{0^{\mu}}}}=I-\bar\partial_{\mu,E_{\rho_E^{0^{\mu}}}}\Delta_{0,\mu,E_{\rho_E^{0^{\mu}}}}^{-1}\bar\partial_{\mu,E_{\rho_E^{0^{\mu}}}}^*$ we have the last identity.
\end{proof} 
\begin{prop}\label{prop:KSfib}
    The Kodaira-Spencer map of the curve $\rho_E^{(\nu+t\tilde\nu)^{\mu+t\tilde\mu}}$ at $t=0$ is
  \begin{align*}
    KS_{\nu^\mu}(\tilde\mu\oplus\tilde\nu)=& P_{\mu}\tilde\mu^\mu\oplus P^{0,1}_{\nu^\mu}\left(\text{Ad}(f^{\nu^{\mu}}) ((f^{\nu^{\mu}})^{-1}\cdot (\partial f^{\nu^{\mu}})\tilde\mu^\mu+\tilde\nu^\mu)\right.\\ +\text{Ad}f^{\nu^\mu}&\left.\left.(\Phi_1^{\mu})_*^{-1}
 \frac{d}{dt}\vert_{t=0} (\Phi_1^{\mu+t\tilde\mu})_*P_{\mu+t\tilde\mu}^{0,1} (\Phi_1^{\mu+t\tilde\mu})_*^{-1}(\nu)) (1-\vert\mu\circ(\Phi_1^\mu)^{-1}\vert^2) 
\right)
\right)
  \end{align*}
with $\tilde\mu^\mu=(\frac{\tilde\mu}{1-\vert\mu\vert^2}\frac{\partial \Phi_1^\mu}{\overline{\partial \Phi_1^\mu}})\circ (\Phi_1^\mu)^{-1}$ and $P_{\mu}$ and $P^{0,1}_{\nu^\mu}$ the $L^2$-projections on the harmonic forms $\mathcal{H}^{0,1}(X_{\rho_\mu},TX_{\rho_\mu})$ respectivly $\mathcal{H}^{0,1}(X_{\rho_\mu},\textnormal{End}E_{\rho_E^{\nu^\mu}})$
\end{prop}
\begin{proof}
First, we observe that the Teichm{\"u}ller direction is unchanged from the classical case. Now we want to use Lemma \ref{lem:easecomp}, and so using that $\rho_E^{(\nu+t\tilde\nu)^{\mu+t\tilde\mu}}$ is independent of $z$ we find that
\begin{align*}
  \rho_E^{(\nu+t\tilde\nu)^{\mu+t\tilde\mu}}(\gamma)&=f^{(\nu+t\tilde\nu)^{\mu+t\tilde\mu}}(\rho_{\mu+t\tilde\mu}(\gamma)z)\rho_E(\gamma) (f^{(\nu+t\tilde\nu)^{\mu+t\tilde\mu}}(z))^{-1}\\ &=f^{(\nu+t\tilde\nu)^{\mu+t\tilde\mu}}((\Phi_1^{\mu+t\tilde\mu}((\Phi_1^\mu)^{-1}(\rho_{\mu}(\gamma)z)))\rho_E(\gamma)\\ &\qquad (f^{(\nu+t\tilde\nu)^{\mu+t\tilde\mu}}((\Phi_1^{\mu+t\tilde\mu}((\Phi_1^\mu)^{-1}(z)))^{-1}.
\end{align*}
Next we have to calculate
\begin{align}
 \text{Ad}(f^{\nu^{\mu}}) \frac{d}{dt}&\vert_{t=0}\left(\left((f^{{(\nu+t\tilde\nu)^{\mu+t\tilde\mu}}}\circ \Phi_1^{\mu+t \tilde\mu})\circ  (\Phi_1^\mu)^{-1}\right)\bar\partial \left((f^{\nu^{\mu+t \tilde\mu}})^{-1}\circ \Phi_1^{\mu+t \tilde\mu}\circ  (\Phi_1^\mu)^{-1}\right)\right)\nonumber \\ \nonumber =&\text{Ad}(f^{\nu^{\mu}}) \frac{d}{dt}\vert_{t=0}\left(\left({(\nu+t\tilde\nu)^{\mu+t\tilde\mu}})\circ \Phi_1^{\mu+t \tilde\mu}\circ  (\Phi_1^\mu)^{-1}\right)(\overline{\partial(\Phi_1^{\mu+t \tilde\mu}\circ  (\Phi_1^\mu)^{-1})})\right)
\nonumber \\ \nonumber &+\text{Ad}(f^{\nu^{\mu}}) \frac{d}{dt}\vert_{t=0}\Big(\left(f^{{(\nu+t\tilde\nu)^{\mu+t\tilde\mu}}}\circ \Phi_1^{\mu+t \tilde\mu}\circ  (\Phi_1^\mu)^{-1}\right)^{-1}
\nonumber \\ \nonumber &\qquad\cdot (\partial f^{\nu^{\mu}})\circ \Phi_1^{\mu+t \tilde\mu}\circ  (\Phi_1^\mu)^{-1}(\overline{\partial}(\Phi_1^{\mu+t \tilde\mu}\circ  (\Phi_1^\mu)^{-1}))\Big).
\end{align}
For the first term we find that
\begin{align}
\text{Ad}(f^{\nu^{\mu}}) \frac{d}{dt}\vert_{t=0}&\left(\left({(\nu+t\tilde\nu)^{\mu+t\tilde\mu}})\circ \Phi_1^{\mu+t \tilde\mu}\circ  (\Phi_1^\mu)^{-1}\right)(\overline{\partial(\Phi_1^{\mu+t \tilde\mu}\circ  (\Phi_1^\mu)^{-1})})\right)=\nonumber \\ \nonumber =& \frac{d}{dt}\vert_{t=0}(\text{Ad}(f^{\nu^{\mu}})({(\nu+t\tilde\nu)^{\mu+t\tilde\mu}})\circ \Phi_1^{\mu+t \tilde\mu}\circ  (\Phi_1^\mu)^{-1} \nonumber \\ \nonumber &\cdot \overline{\left((\partial\Phi_1^{\mu+t \tilde\mu}\circ  (\Phi_1^\mu)^{-1})\partial(\Phi_1^\mu)^{-1}+(\bar\partial\Phi_1^{\mu+t \tilde\mu}\circ  (\Phi_1^\mu)^{-1})\partial(\bar\Phi_1^\mu)^{-1}\right)} .
\nonumber 
\end{align}
We can now rewrite the last factor using (\ref{eq:andet led}) and (\ref{eq:overlines}) and their conjugates to get that
\begin{align*}
(\partial\Phi_1^{\mu+t \tilde\mu}\circ  &(\Phi_1^\mu)^{-1})\partial(\Phi_1^\mu)^{-1}+(\bar\partial\Phi_1^{\mu+t \tilde\mu}\circ  (\Phi_1^\mu)^{-1})\partial(\bar\Phi_1^\mu)^{-1}\\ &=(\partial\Phi_1^{\mu+t \tilde\mu}\circ  (\Phi_1^\mu)^{-1})\partial(\Phi_1^\mu)^{-1} \\ &\quad+(((\mu+t \tilde\mu)\partial\Phi_1^{\mu+t \tilde\mu})\circ  (\Phi_1^\mu)^{-1}) (-\bar\mu\circ(\Phi_1^{\mu})^{-1}\partial(\Phi_1^{\mu})^{-1}).\\ &=(\partial\Phi_1^{\mu+t \tilde\mu}\circ  (\Phi_1^\mu)^{-1})\partial(\Phi_1^\mu)^{-1} (1-(((\mu+t \tilde\mu)\bar\mu)\circ  (\Phi_1^\mu)^{-1})).\end{align*}
Using that $(\Phi_1^\mu)_*(\nu)=\nu\circ \Phi_1^\mu \overline{\partial \Phi_1^\mu}$ in the first term we find that
\begin{align*}
  \text{Ad}&(f^{\nu^{\mu}}) \frac{d}{dt}\vert_{t=0}\left(\left({(\nu+t\tilde\nu)^{\mu+t\tilde\mu}})\circ \Phi_1^{\mu+t \tilde\mu}\circ  (\Phi_1^\mu)^{-1}\right)(\overline{\partial(\Phi_1^{\mu+t \tilde\mu}\circ  (\Phi_1^\mu)^{-1})})\right)\\ &=\text{Ad}(f^{\nu^{\mu}})\frac{d}{dt}\vert_{t=0}\left( (\Phi_1^\mu)^{-1}_*\left((\Phi_1^{\mu+t \tilde\mu})_*\left({(\nu+t\tilde\nu)^{\mu+t\tilde\mu}}\right)(1-\vert\mu\vert^2-\bar t\mu\overline{\tilde\mu})  \right)\right)\\ &= \text{Ad}(f^{\nu^{\mu}})(\Phi_1^\mu)^{-1}_*\left((1-\vert\mu\vert^2) \frac{d}{dt}\vert_{t=0}(\Phi_1^{\mu+t \tilde\mu})_*\left({P^{0,1}_{\mu+t\tilde\mu,E_{\rho_E^{0^{\mu+t\tilde\mu}}}
}(\Phi_1^{\mu+t \tilde\mu})_*^{-1}(\nu+t\tilde\nu)}\right)\right)\\ &= \text{Ad}(f^{\nu^{\mu}}) P^{0,1}_{\mu,E_{\rho_E^{0^{\mu}}}}\left((\Phi_1^\mu)^{-1}_*(\tilde\nu)+\tilde\mu^\mu\partial_{\mu,E_{\rho_E^{0^{\mu}}}}\Delta^{-1}_{0,\mu,E_{\rho_E^{0^{\mu}}}}\bar\partial_{\mu,E_{\rho_E^{0^{\mu}}}}^* (\Phi_1^\mu)^{-1}_*(\nu)\right)(\Phi_1^\mu)^{-1}_*\\ &\hspace{10cm}\left((1-\vert\mu\vert^2)\right).
\end{align*} Here we have used the result from Lemma \ref{lem:projderiv} to calculate the derivative of the projection.

For the second term we rewrite
\begin{align}
\overline{\partial}&(\Phi_1^{\mu+t \tilde\mu}\circ  (\Phi_1^\mu)^{-1})\nonumber\\&=(\bar\partial\Phi_1^{\mu+t \tilde\mu})\circ  (\Phi_1^\mu)^{-1}\bar\partial\overline{(\Phi_1^{\mu})^{-1}}+(\partial\Phi_1^{\mu+t \tilde\mu})\circ  (\Phi_1^\mu)^{-1}\bar\partial(\Phi_1^{\mu})^{-1}\nonumber\\ \label{eq:bers diff} &=((\mu+t \tilde\mu)\partial\Phi_1^{\mu+t \tilde\mu})\circ  (\Phi_1^\mu)^{-1}\bar\partial\overline{(\Phi_1^{\mu})^{-1}}+(\partial\Phi_1^{\mu+t \tilde\mu})\circ  (\Phi_1^\mu)^{-1}\bar\partial(\Phi_1^{\mu})^{-1}
\end{align}
using \eqref{eq:overlines} and \eqref{eq:andet led} in \eqref{eq:bers diff} and find that
\begin{align*}
  \overline{\partial}(\Phi_1^{\mu+t \tilde\mu}\circ  (\Phi_1^\mu)^{-1})=\left(\frac{t\tilde\mu}{1-\vert\mu\vert^2}\frac{\partial \Phi_1^{\mu+t\tilde\mu}}{\overline{\partial \Phi_1^\mu}}\right)\circ (\Phi_1^\mu)^{-1},
\end{align*}
which implies that
\begin{align*}\text{Ad}(f^{\nu^{\mu}}) \frac{d}{dt}\vert_{t=0}&\Big(\left(f^{{(\nu+t\tilde\nu)^{\mu+t\tilde\mu}}}\circ \Phi_1^{\mu+t \tilde\mu}\circ  (\Phi_1^\mu)^{-1}\right)^{-1}
 \\  &\qquad\cdot (\partial f^{\nu^{\mu}})\circ \Phi_1^{\mu+t \tilde\mu}\circ  (\Phi_1^\mu)^{-1}(\overline{\partial}(\Phi_1^{\mu+t \tilde\mu}\circ  (\Phi_1^\mu)^{-1}))\Big)\\ &=\text{Ad}(f^{\nu^{\mu}}) ((f^{\nu^{\mu}})^{-1}\cdot (\partial f^{\nu^{\mu}})(\tilde\mu^\mu),
\end{align*}
where $\tilde\mu^\mu=\left(\frac{\tilde\mu}{1-\vert\mu\vert^2}\frac{\partial \Phi_1^\mu}{\overline{\partial \Phi_1^\mu}}\right)\circ (\Phi_1^\mu)^{-1}$. And so we have that
\begin{align}
  \label{eq:fibcoordKS1}
 \text{Ad}&(f^{\nu^{\mu}}) \frac{d}{dt}\vert_{t=0}\left(\left((f^{{(\nu+t\tilde\nu)^{\mu+t\tilde\mu}}}\circ \Phi_1^{\mu+t \tilde\mu})\circ  (\Phi_1^\mu)^{-1}\right)\bar\partial \left((f^{\nu^{\mu+t \tilde\mu}})^{-1}\circ \Phi_1^{\mu+t \tilde\mu}\circ  (\Phi_1^\mu)^{-1}\right)\right)\nonumber
\\ \nonumber  =& \text{Ad}(f^{\nu^{\mu}})\left((\Phi_1^{\mu})_*^{-1}
 \frac{d}{dt}\vert_{t=0} (\Phi_1^{\mu+t\tilde\mu})_* P_{\mu+t\tilde\mu,E_{\rho_E^{0^{\mu+t\tilde\mu}}}}^{0,1} (\Phi_1^{\mu+t\tilde\mu})_*^{-1}(\nu))))(1-\vert\mu\vert^2)\circ  (\Phi_1^\mu)^{-1})
\right)
\\  &+\text{Ad}(f^{\nu^{\mu}})\tilde\nu^\mu+\text{Ad}(f^{\nu^{\mu}}) ((f^{\nu^{\mu}})^{-1}\cdot (\partial f^{\nu^{\mu}})(\tilde\mu^\mu)).
\end{align}
We have thus shown that composing with the projection gives us the harmonic representative.
\end{proof}

\subsection{Comparison of the Two Tangent Maps and a proof of the first part of Theorem  \ref{thm:secondordvar}}

We compare $$P^{0,1}_{\mu\oplus\nu, 
}\left((\Phi_1^{\mu})_*^{-1}\left(\text{Ad}\Phi_2^{\mu\oplus\nu}\left(\tilde\mu(\Phi_2^{\mu\oplus\nu})^{-1}\partial\Phi_2^{\mu\oplus\nu}+\tilde\nu\right)\right)\right)$$ and 
\begin{multline*}
  P^{0,1}_{\nu^\mu}\left(\text{Ad}(f^{\nu^{\mu}}) ((f^{\nu^{\mu}})^{-1}\cdot (\partial f^{\nu^{\mu}})\tilde\mu^\mu+\tilde\nu^\mu)\right.\\ +\text{Ad}f^{\nu^\mu}\left.(\Phi_1^{\mu})_*^{-1}
 \frac{d}{dt}\vert_{t=0} (\Phi_1^{\mu+t\tilde\mu})_*P_{\mu+t\tilde\mu,E_{\rho_E^{0^{\mu+t\tilde\mu}}}}^{0,1} (\Phi_1^{\mu+t\tilde\mu})_*^{-1}(\nu)) (1-\vert\mu\circ(\Phi_1^\mu)^{-1}\vert^2)\right).
\end{multline*}
First we observe that $$\text{Ad}f^{\nu^\mu}(\Phi_1^{\mu})_*^{-1}
 \frac{d}{dt}\vert_{t=0} (\Phi_1^{\mu+t\tilde\mu})_*P_{\mu+t\tilde\mu,E_{\rho_E^{0^{\mu+t\tilde\mu}}}}^{0,1} (\Phi_1^{\mu+t\tilde\mu})_*^{-1}(\nu)) (1-\vert\mu\circ(\Phi_1^\mu)^{-1}\vert^2)$$ 
vanishes to first order in $\nu$ and $\mu$ at the center, since we either differentiate with respect to $\mu$ and set $\nu=0$ or we differentiate with respect to $\nu$ and then we find, when we evaluate at $\mu=0$, that $\bar\partial_{0,E}^* \nu=0$, from the expression in Lemma \ref{lem:projderiv}. 

Next we compare $$(\Phi_1^{\mu})_*^{-1}\left(\text{Ad}\Phi_2^{\mu\oplus\nu}\left(\tilde\mu(\Phi_2^{\mu\oplus\nu})^{-1}\partial\Phi_2^{\mu\oplus\nu}\right)\right)$$ with $$\text{Ad}(f^{\nu^{\mu}}) ((f^{\nu^{\mu}})^{-1}\cdot (\partial f^{\nu^{\mu}})\tilde\mu^\mu).$$
We observe, that since $\partial I=0$ both $(\Phi_2^{\mu\oplus\nu})^{-1}\partial\Phi_2^{\mu\oplus\nu}$ and $(f^{\nu^{\mu}})^{-1}\cdot (\partial f^{\nu^{\mu}})$ vanish unless we differentiate it with respect to the moduli space direction or the Teichm{\"u}ller direction. If we differentiate with respect to $\mu$ we get $\frac{\partial}{\partial \epsilon}\nu^{\epsilon\mu}$, but at $\nu=0$ this is $0$. This means we can compare the two after evaluating $\mu=0$, and then we have $f^{\nu^0}=\Phi^{0\oplus\nu}$, and so they agree to first order.

The last terms to consider are $(\Phi_1^{\mu})_*^{-1}\left(\text{Ad}\Phi_2^{\mu\oplus\nu}(\tilde\nu)\right)$ and $\text{Ad}(f^{\nu^{\mu}})\tilde\nu^\mu$. Now, if we put $\mu=0$ the terms agree. If we differentiate with respect to $\mu$, we can put $\nu=0$ first. We are differentiating a term of the form $\bar\partial_{\mu,E_{\rho_E^{0^\mu}}}\Delta_{0,\mu,E_{\rho_E^{0^\mu}}}^{-1}\bar\partial_{\mu,E_{\rho_E^{0^\mu}}}^*(\Phi_1^{\mu})_*^{-1}\tilde\nu$ with respect to $\mu$. The result is an exact term which is killed by the harmonic projection $P^{0,1}$, plus a term containing $\bar\partial_{0,E}^*\nu=0$. This proves the first part of Theorem \ref{thm:secondordvar}. The second part will be proved in the following section.

\section{Variation of the Metric}
\label{sec:metvaruniv}

In order to prove that our new coordinates are not the same as the fibered coordinates discussed above, we shall consider the variation of the metric in both set of coordinates and use the resulting formulae to demonstrate that they are not identical to third order.

\subsection{Variation in the Universal Coordinates}

In this section will calculate the second variation of the metric using the coordinates from Theorem \ref{thm:uni coord}. In the next section we will do the same for the fibered coordinates, and use this to show that the two sets of coordinates differ at third order.
So first we consider the function $(\bar\Phi_2^{\epsilon(\mu\oplus\nu)})^T\Phi^{\epsilon(\mu\oplus\nu)}_2$. This transforms as a function on $X$ with values in $\text{End}E$, our reference point. Now, to further understand this function, we look at $\frac{d}{d\epsilon}\vert_{\epsilon=0}(\bar\Phi_2^{\epsilon(\mu\oplus\nu)})^T\Phi^{\epsilon(\mu\oplus\nu)}_2$. Then we find that
\begin{align*}
  \Delta\frac{d}{d\epsilon}\vert_{\epsilon=0}(\bar\Phi_2^{\epsilon(\mu\oplus\nu)})^T\Phi^{\epsilon(\mu\oplus\nu)}_2=\Delta (\frac{d}{d\epsilon}\vert_{\epsilon=0}(&\bar\Phi_+^{\epsilon(\mu\oplus\nu)})^T+\frac{d}{d\epsilon}\vert_{\epsilon=0}\Phi^{\epsilon(\mu\oplus\nu)}_-\\ &+\frac{d}{d\epsilon}\vert_{\epsilon=0}(\bar\Phi_-^{\epsilon(\mu\oplus\nu)})^T+\frac{d}{d\epsilon}\vert_{\epsilon=0}\Phi^{\epsilon(\mu\oplus\nu)}_+),
\end{align*}
and since $\Phi^{\epsilon(\mu\oplus\nu)}_-$ is antiholomorphic we get that
\begin{align*}
  \Delta\frac{d}{d\epsilon}\vert_{\epsilon=0}(\bar\Phi_2^{\epsilon(\mu\oplus\nu)})^T\Phi^{\epsilon(\mu\oplus\nu)}_2=\frac{d}{d\epsilon}\vert_{\epsilon=0}\Delta(\bar\Phi_+^{\epsilon(\mu\oplus\nu)})^T+\frac{d}{d\epsilon}\vert_{\epsilon=0}\Delta\Phi^{\epsilon(\mu\oplus\nu)}_+).
\end{align*}
We now use that $(\bar\partial-\epsilon\mu\partial)\Phi^{\epsilon(\mu\oplus\nu)}_+=0$ and $\Delta=y^{-2}\partial\bar\partial$ to see that 
\begin{align*}
  \Delta\frac{d}{d\epsilon}\vert_{\epsilon=0}(\bar\Phi_2^{\epsilon(\mu\oplus\nu)})^T\Phi^{\epsilon(\mu\oplus\nu)}_2=y^{-2}\bar\mu\bar\partial\bar\partial(\bar\Phi_+^{0(\mu\oplus\nu)})^T+y^{-2}\mu\partial\partial\Phi^{0(\mu\oplus\nu)}_+)=0,
\end{align*}
since $\Phi_+^0=I$, and so the derivative is $0$. This allows us to conclude, that $\frac{d}{d\epsilon}\vert_{\epsilon=0}(\bar\Phi_2^{\epsilon(\mu\oplus\nu)})^T\Phi^{\epsilon(\mu\oplus\nu)}_2$ is a constant multiple of the identity element in End$E$, and because of the determinant criteria in Theorem~\ref{thm:uni coord} we have \begin{align*}0&=\frac{d}{d\epsilon}\vert_{\epsilon=0}(\det(\bar\Phi_2^{\epsilon(\mu\oplus\nu)})^T\Phi^{\epsilon(\mu\oplus\nu)}_2)\\ &=\tr\vert_{\epsilon=0}((\bar\Phi_2^{0})^T\Phi^{0}_2)^{-1}\frac{d}{d\epsilon}\vert_{\epsilon=0}(\bar\Phi_2^{\epsilon(\mu\oplus\nu)})^T\Phi^{\epsilon(\mu\oplus\nu)}_2=\tr \frac{d}{d\epsilon}\vert_{\epsilon=0}(\bar\Phi_2^{\epsilon(\mu\oplus\nu)})^T\Phi^{\epsilon(\mu\oplus\nu)}_2, \end{align*}
and so $\frac{d}{d\epsilon}\vert_{\epsilon=0}(\bar\Phi_2^{\epsilon(\mu\oplus\nu)})^T\Phi^{\epsilon(\mu\oplus\nu)}_2=0$. We see that this immediately implies that $\frac{d}{d\epsilon}\vert_{\epsilon=0}\partial \Phi_+^{\epsilon(\mu\oplus\nu)}=-\overline{\frac{d}{d\bar\epsilon}\vert_{\epsilon=0}\bar\partial \Phi_-^{\epsilon(\mu\oplus\nu)}}^T=0$. We can study $$\frac{d}{d\bar\epsilon}\vert_{\epsilon=0}(\bar\Phi_2^{\epsilon(\mu\oplus\nu)})^T\Phi^{\epsilon(\mu\oplus\nu)}_2$$ similarly and conclude that $$\frac{d}{d\bar\epsilon}\vert_{\epsilon=0}\partial \Phi_+^{\epsilon(\mu\oplus\nu)}=-\overline{\frac{d}{d\epsilon}\vert_{\epsilon=0}\bar\partial \Phi_-^{\epsilon(\mu\oplus\nu)}}^T=-\bar\nu^T.$$

Now, we want to understand the variation of the $\bar\partial_{\mu,E_{\rho^{(\mu\oplus\nu)}}}$-operator on functions and $\bar\partial_{\mu,E_{\rho^{(\mu\oplus\nu)}}}^*$-operator on $(0,1)$-forms, since they play a central role in understanding the tangent spaces over the universal moduli space. We work on the universal cover and pull back our family of differential operators from the universal cover of $(X_{\rho_\mu},E_{\rho^{(\mu\oplus\nu)}})$ to that of $(X_{\rho_0},E)$, in terms of representations. Then $\bar\partial_{\mu,E_{\rho^{(\mu\oplus\nu)}}}$ is just represented by $\bar\partial$ on $\SetH$
\begin{align}
  \frac{d}{d\epsilon}\vert_{\epsilon=0}\text{Ad}&\Phi_2^{\epsilon \mu\oplus\nu}(\Phi_1^{\epsilon \mu})_*\bar\partial (\Phi_1^{\epsilon \mu})_*^{-1}(\text{Ad}\Phi_2^{\epsilon \mu\oplus\nu})^{-1}\nonumber\\\nonumber &=\frac{d}{d\epsilon}\vert_{\epsilon=0}\text{Ad}\Phi_2^{\epsilon \mu\oplus\nu}\frac{1}{1-\vert\epsilon\mu\vert^2}(\bar\partial -\mu\partial)(\text{Ad}\Phi_2^{\epsilon \mu\oplus\nu})^{-1} \\ &=\frac{d}{d\epsilon}\vert_{\epsilon=0}\frac{1}{1-\vert\epsilon\mu\vert^2}\big(\text{Ad}\Phi_2^{\epsilon \mu\oplus\nu}(\text{ad}(\bar\partial -\epsilon\mu\partial)\Phi_2^{\epsilon \mu\oplus\nu})(\text{Ad}\Phi_2^{\epsilon \mu\oplus\nu})^{-1}\nonumber\\ &\hspace{7cm}+(\bar\partial -\epsilon\mu\partial)\big)\nonumber\\\nonumber &=\frac{d}{d\epsilon}\vert_{\epsilon=0}\frac{1}{1-\vert\epsilon\mu\vert^2}(\epsilon\text{ad}\text{Ad}\Phi_2^{\epsilon \mu\oplus\nu}\nu+(\bar\partial -\epsilon\mu\partial))\\ &=\text{ad}\nu-\mu\partial. \label{eq:barparvarunicoord}
\end{align}

Likewise we find that the variation of $\bar\partial^*=-\rho^{-1}\partial$,
 where also the first derivative of density $\rho$ is zero at the center point of our coordinates(\cite{Wolpert}). We begin by observing that on $(0,1)$-forms we have that
\begin{align*}
(\Phi_1^{\epsilon \mu})_*\partial (\Phi_1^{\epsilon \mu})_*^{-1}\alpha &=(\Phi_1^{\epsilon \mu})_* \partial(\alpha\circ (\Phi_1^{\epsilon\mu})^{-1}\frac{1}{\bar\partial\overline{\Phi_1^{\epsilon\mu}\circ (\Phi_1^{\epsilon\mu})^{-1}}})\\ &=(\Phi_1^{\epsilon \mu})_* ((\partial\alpha)\circ (\Phi_1^{\epsilon\mu})^{-1}\frac{\partial(\Phi_1^{\epsilon\mu})^{-1}}{\bar\partial\overline{\Phi_1^{\epsilon\mu}\circ (\Phi_1^{\epsilon\mu})^{-1}}}\\ &+(\bar\partial\alpha)\circ (\Phi_1^{\epsilon\mu})^{-1}\frac{\partial\overline{(\Phi_1^{\epsilon\mu})^{-1}}}{\bar\partial\overline{\Phi_1^{\epsilon\mu}\circ (\Phi_1^{\epsilon\mu})^{-1}}}\\ &-\alpha\circ (\Phi_1^{\epsilon\mu})^{-1}\frac{\overline{(\partial(\bar\partial\Phi_1^{\epsilon\mu}))\circ (\Phi_1^{\epsilon\mu})^{-1}  (\bar\partial\bar\Phi_1^{\epsilon\mu})^{-1}}}{(\bar\partial\overline{\Phi_1^{\epsilon\mu}\circ (\Phi_1^{\epsilon\mu})^{-1}})^2})\\ &-\alpha\circ (\Phi_1^{\epsilon\mu})^{-1}\frac{\overline{(\partial\partial\Phi_1^{\epsilon\mu})\circ (\Phi_1^{\epsilon\mu})^{-1}  (\bar\partial\Phi_1^{\epsilon\mu})^{-1}}}{(\bar\partial\overline{\Phi_1^{\epsilon\mu}\circ (\Phi_1^{\epsilon\mu})^{-1}})^2})\\ &=\frac{1}{1-\vert\epsilon\mu\vert^2}(\partial-\bar\epsilon\bar\mu\bar\partial-\bar\epsilon(\bar\partial\bar\mu))=\frac{1}{1-\vert\epsilon\mu\vert^2}(\partial-\overline{\epsilon\partial\mu})).
\end{align*}
And so we find that
\begin{align} 
  \frac{d}{d\bar\epsilon}\vert_{\epsilon=0}\text{Ad}&\Phi_2^{\epsilon \mu\oplus\nu}(\Phi_1^{\epsilon \mu})_*\bar\partial^* (\Phi_1^{\epsilon \mu})_*^{-1}(\text{Ad}\Phi_2^{\epsilon \mu\oplus\nu})^{-1} \nonumber\\ \nonumber&=\frac{d}{d\bar\epsilon}\vert_{\epsilon=0}\text{Ad}\Phi_2^{\epsilon \mu\oplus\nu}\frac{-\rho^{-1}}{1-\vert\epsilon\mu\vert^2}(\partial-\bar\epsilon\bar\partial \bar\mu)(\text{Ad}\Phi_2^{\epsilon \mu\oplus\nu})^{-1} \\ &=\frac{d}{d\bar\epsilon}\vert_{\epsilon=0}\frac{-\rho^{-1}}{1-\vert\epsilon\mu\vert^2}\big(\text{Ad}\Phi_2^{\epsilon \mu\oplus\nu}(\text{Ad}(\partial-\bar\epsilon\bar\partial \bar\mu)\Phi_2^{\epsilon \mu\oplus\nu})^{-1}\nonumber\\ \nonumber&\hspace{4cm}+(\bar\partial -\bar\epsilon\bar\partial\bar\mu)\big)\\ &=\frac{d}{d\bar\epsilon}\vert_{\epsilon=0}\frac{-\rho^{-1}}{1-\vert\epsilon\mu\vert^2}(\text{ad}(\text{Ad}(\Phi_2^{\epsilon(\mu\oplus\nu)})^{-1}\partial\Phi_2^{\epsilon(\mu\oplus\nu)})-\bar\epsilon\bar\mu\rho^{-1}\text{ad}\bar\partial\Phi_2^{\epsilon(\mu\oplus\nu)}\nonumber\\ \nonumber&\hspace{4cm}+(\partial-\bar\epsilon\bar\partial \bar\mu))\\ &=-*\text{ad}\nu*+\bar\mu\bar\partial\rho^{-1},\label{eq:delbarstjerneunicoord}
\end{align}
where the equality follows from the equation $\partial\mu=2y^{-1}\mu$, and $\rho^{-1}=y^2$.

This is the first step in understanding the metric on the universal moduli space of pairs of a Riemann surface and a holomorphic bundle over it, given at a point $(X,E)$ by identifying the tangent space with $\mathcal{H}^{0,1}(X,TX)\oplus\mathcal{H}^{0,1}(X,\text{End}E)$. Two elements $\mu_1\oplus\nu_1$ and $\mu_2\oplus\nu_2$ can be paired as follows
\begin{align*}
 g(\mu_1\oplus\nu_1,\mu_2\oplus\nu_2) =\int_\Sigma( \rho_X\mu_1\bar\mu_2+i\tr\nu_1\wedge\star\bar\nu_2^T),
\end{align*}
where $\rho_X$ is the density of the hyperbolic metric corresponding to the complex structure on $X$. Since the term $\int_\Sigma \rho_X\mu_1\bar\mu_2$, is independent of the bundle, nothing has changed compared to the situation on Teichm{\"u}ller space. Let us examine the term $\int_\Sigma \tr\nu_1\wedge\bar(-\star)\nu_2^T$. Since we are evaluation the metric on tangent vectors, $-\star$ will act by $-i$ and so we replace it in the following to avoid confusion.

 In coordinates around $(X,E)$ we have, using Proposition \ref{prop:KSuniv}, that the metric is given by
\begin{align}
  \nonumber
  g^{VB}_{\epsilon(\mu\oplus\nu)}(&\mu_1\oplus\nu_1,\mu_2\oplus\nu_2)\\ \nonumber&=
                             - i                                                   \int_\Sigma \tr P^{0,1}_{\epsilon(\mu\oplus\nu)}(\Phi_1^{\epsilon(\mu\oplus\nu)})_*^{-1}\text{Ad}(\Phi_2^{\epsilon(\mu\oplus\nu)})\nu_1\\ &\hspace{5cm} \nonumber\wedge \overline{P^{0,1}_{\epsilon(\mu\oplus\nu)}((\Phi_1^{\epsilon(\mu\oplus\nu)})_*^{-1}\text{Ad}\Phi_2^{\epsilon(\mu\oplus\nu)})\nu_2}^T
  \\\nonumber &-i \int_\Sigma \tr P^{0,1}_{\epsilon(\mu\oplus\nu)}(\Phi_1^{\epsilon(\mu\oplus\nu)})_*^{-1}\text{Ad}(\Phi_2^{\epsilon(\mu\oplus\nu)})\mu_1(\Phi_2^{\epsilon(\mu\oplus\nu)})^{-1}\partial\Phi_2^{\epsilon(\mu\oplus\nu)}\\ \nonumber&\hspace{5cm}\wedge \overline{P^{0,1}_{\epsilon(\mu\oplus\nu)}((\Phi_1^{\epsilon(\mu\oplus\nu)})_*^{-1}\text{Ad}\Phi_2^{\epsilon(\mu\oplus\nu)})\nu_2}^T
  \\\nonumber &-i\int_\Sigma \tr P^{0,1}_{\epsilon(\mu\oplus\nu)}(\Phi_1^{\epsilon(\mu\oplus\nu)})_*^{-1}\text{Ad}(\Phi_2^{\epsilon(\mu\oplus\nu)})\nu_1\\\nonumber &\hspace{2cm}\wedge \overline{P^{0,1}_{\epsilon(\mu\oplus\nu)}(\Phi_1^{\epsilon(\mu\oplus\nu)})_*^{-1}(\text{Ad}\Phi_2^{\epsilon(\mu\oplus\nu)})\mu_2(\Phi_2^{\epsilon(\mu\oplus\nu)})\partial\Phi_2^{\epsilon(\mu\oplus\nu)}}^T
  \\\nonumber &-i\int_\Sigma \tr P^{0,1}_{\epsilon(\mu\oplus\nu)}(\Phi_1^{\epsilon(\mu\oplus\nu)})_*^{-1}\text{Ad}(\Phi_2^{\epsilon(\mu\oplus\nu)})\mu_1(\Phi_2^{\epsilon(\mu\oplus\nu)})^{-1}\partial\Phi_2^{\epsilon(\mu\oplus\nu)} \\ \label{eq:metuniv}&\hspace{1.5cm} \wedge \overline{P^{0,1}_{\epsilon(\mu\oplus\nu)}((\Phi_1^{\epsilon(\mu\oplus\nu)})_*^{-1}\text{Ad}\Phi_2^{\epsilon(\mu\oplus\nu)})\mu_2(\Phi_2^{\epsilon(\mu\oplus\nu)})^{-1}\partial\Phi_2^{\epsilon(\mu\oplus\nu)}}^T.
\end{align}
Now we can use that $P^{0,1}_{\epsilon(\mu\oplus\nu)}$ is self-adjoint with respect to the metric to rewrite the terms as follows
\begin{multline}
\label{eq:termsgath}
   \int_\Sigma \tr P^{0,1}_{\epsilon(\mu\oplus\nu)}(\Phi_1^{\epsilon(\mu\oplus\nu)})_*^{-1}\text{Ad}(\Phi_2^{\epsilon(\mu\oplus\nu)})\nu_1\wedge \overline{P^{0,1}_{\epsilon(\mu\oplus\nu)}((\Phi_1^{\epsilon(\mu\oplus\nu)})_*^{-1}\text{Ad}\Phi_2^{\epsilon(\mu\oplus\nu)})\nu_2}^T \\= \int_\Sigma \tr \text{Ad}(\overline{\Phi_2^{\epsilon(\mu\oplus\nu)}}^T\Phi_2^{\epsilon(\mu\oplus\nu)})(1-\vert\epsilon\mu\vert^2)\text{Ad}(\Phi_2^{\epsilon(\mu\oplus\nu)})^{-1}\\(\Phi_1^{\epsilon(\mu\oplus\nu)})_*P^{0,1}_{\epsilon(\mu\oplus\nu)}(\Phi_1^{\epsilon(\mu\oplus\nu)})_*^{-1}\text{Ad}(\Phi_2^{\epsilon(\mu\oplus\nu)})\nu_1\wedge  \overline{\nu_2}^T
\end{multline}
Since $(\Phi_1^{\epsilon\mu})_*(d\bar z\wedge dz)=(\vert \partial\Phi_1^{\epsilon\mu}\vert^2-\vert\bar\partial\Phi_1^{\epsilon\mu}\vert^2)d\bar z\wedge dz$. Also recall that $$(\Phi_1^{\epsilon\mu})_*^{-1} \nu=\left(\frac{\nu}{\bar\partial\bar\Phi_1^{\epsilon\mu}}\right)\circ(\Phi_1^{\epsilon\mu})^{-1}$$ and $(\Phi_1^{\epsilon\mu})_* P^{0,1}_{\epsilon(\mu\oplus\nu)} h= (\bar\partial \bar\Phi_1^{\epsilon\mu}) (P^{0,1}_{\epsilon(\mu\oplus\nu)}h)\circ\Phi_1^{\epsilon\mu}$.

From this it follows that
\begin{lemma}
  In the coordinates around $(X,E)$ given by Theorem~\ref{thm:uni coord} we have that
  \begin{align*}
    \frac{d}{d\epsilon}\vert_{\epsilon=0}g^{VB}_{\epsilon(\mu\oplus\nu)}(&\mu_1\oplus\nu_1,\mu_2\oplus\nu_2)=i\int_X\tr((\bar\mu_2\nu_1) \wedge\nu)\\  \frac{d}{d\bar\epsilon}\vert_{\epsilon=0}g^{VB}_{\epsilon(\mu\oplus\nu)}(&\mu_1\oplus\nu_1,\mu_2\oplus\nu_2)=i\int_X\tr((\mu_1\bar\nu^T) \wedge\bar\nu_2^T).
  \end{align*}
\end{lemma}
\begin{proof}
We calculate each term gathering the terms like \eqref{eq:termsgath}. We have already seen that $\frac{d}{d\epsilon}\vert_{\epsilon=0}\overline{\Phi_2^{\epsilon(\mu\oplus\nu)}}^T\Phi_2^{\epsilon(\mu\oplus\nu)}=0$, and so these terms don't contribute. Now we consider the operators, where we have left out subscripts from the calculation as it should be clear where they live. The derivative of the projection is a sum of terms starting with an operator ending with $\bar\partial^*$ and ones which starts with $\bar\partial$ as is seen from the following calculation
\begin{align*}
\frac{d}{d\epsilon}\vert_{\epsilon=0}&  \text{Ad}(\Phi_2^{\epsilon(\mu\oplus\nu)})^{-1}(\Phi_1^{\epsilon(\mu\oplus\nu)})_*P^{0,1}(\Phi_1^{\epsilon(\mu\oplus\nu)})_*^{-1}\text{Ad}(\Phi_2^{\epsilon(\mu\oplus\nu)})\\ &=\frac{d}{d\epsilon}\vert_{\epsilon=0}  \text{Ad}(\Phi_2^{\epsilon(\mu\oplus\nu)})^{-1}(\Phi_1^{\epsilon(\mu\oplus\nu)})_*(-\bar\partial\Delta_0^{-1}\bar\partial^*)(\Phi_1^{\epsilon(\mu\oplus\nu)})_*^{-1}\text{Ad}(\Phi_2^{\epsilon(\mu\oplus\nu)})\\ &=\frac{d}{d\epsilon}\vert_{\epsilon=0}  \text{Ad}(\Phi_2^{\epsilon(\mu\oplus\nu)})^{-1}(\Phi_1^{\epsilon(\mu\oplus\nu)})_*(-\bar\partial)(\Phi_1^{\epsilon(\mu\oplus\nu)})_*^{-1}\text{Ad}(\Phi_2^{\epsilon(\mu\oplus\nu)})\\ &\hspace{1cm}\text{Ad}(\Phi_2^{\epsilon(\mu\oplus\nu)})^{-1}(\Phi_1^{\epsilon(\mu\oplus\nu)})_*(\Delta_0^{-1})(\Phi_1^{\epsilon(\mu\oplus\nu)})_*^{-1}\text{Ad}(\Phi_2^{\epsilon(\mu\oplus\nu)})\\ &\hspace{1cm}\text{Ad}(\Phi_2^{\epsilon(\mu\oplus\nu)})^{-1}(\Phi_1^{\epsilon(\mu\oplus\nu)})_*(\bar\partial^*)(\Phi_1^{\epsilon(\mu\oplus\nu)})_*^{-1}\text{Ad}(\Phi_2^{\epsilon(\mu\oplus\nu)})\\&=\frac{d}{d\epsilon}\vert_{\epsilon=0}  \text{Ad}(\Phi_2^{\epsilon(\mu\oplus\nu)})^{-1}(\Phi_1^{\epsilon(\mu\oplus\nu)})_*(-\bar\partial)(\Phi_1^{\epsilon(\mu\oplus\nu)})_*^{-1}\text{Ad}(\Phi_2^{\epsilon(\mu\oplus\nu)})\Delta_0^{-1}\bar\partial^*\\ &+\bar\partial\frac{d}{d\epsilon}\vert_{\epsilon=0}\text{Ad}(\Phi_2^{\epsilon(\mu\oplus\nu)})^{-1}(\Phi_1^{\epsilon(\mu\oplus\nu)})_*(\Delta_0^{-1})(\Phi_1^{\epsilon(\mu\oplus\nu)})_*^{-1}\text{Ad}(\Phi_2^{\epsilon(\mu\oplus\nu)})\bar\partial^*\\ &+\bar\partial\Delta_0^{-1}\frac{d}{d\epsilon}\vert_{\epsilon=0}\text{Ad}(\Phi_2^{\epsilon(\mu\oplus\nu)})^{-1}(\Phi_1^{\epsilon(\mu\oplus\nu)})_*(\bar\partial^*)(\Phi_1^{\epsilon(\mu\oplus\nu)})_*^{-1}\text{Ad}(\Phi_2^{\epsilon(\mu\oplus\nu)}).
\end{align*}
The first two terms are orthogonal to $\nu\in\mathcal{H}^{0,1}(X,\text{End}E)$, and the second one applied to a harmonic from is $0$. This means the contribution form the first term in \eqref{eq:metuniv} is $0$.

Now all the remaining terms contain a $\partial\Phi_2^{\epsilon(\mu\oplus\nu)}$ which is $0$ at $\epsilon=0$. Hence the only contributions to the derivative arise when we derive these, and then we have that $\frac{d}{d\bar\epsilon}\partial\Phi_2^{\epsilon(\mu\oplus\nu)}=-\bar\nu^T$ and $\frac{d}{d\epsilon}\partial\Phi_2^{\epsilon(\mu\oplus\nu)}=0$. Inserting this and setting $\epsilon=0$ we find the formulas in the lemma.
\end{proof}
We proceed to calculate the second order derivatives of the metric. To do so, we need to calculate $\frac{d^2}{d\epsilon_1 d\bar\epsilon_2}\vert_{\epsilon=0}(\overline{\Phi_2^{\epsilon(\mu\oplus\nu)}})^T\Phi_2^{\epsilon(\mu\oplus\nu)}$, $\frac{d^2}{d\epsilon_1 d\bar\epsilon_2}\vert_{\epsilon=0}(\Phi_2^{\epsilon(\mu\oplus\nu)} )^{-1}\partial\Phi_2^{\epsilon(\mu\oplus\nu)}$ and the contribution from $\frac{d^2}{d\epsilon_1 d\bar\epsilon_2}\vert_{\epsilon=0} \text{Ad}(\Phi_2^{\epsilon(\mu\oplus\nu)})^{-1} P^{0,1} \text{Ad}\Phi_2^{\epsilon(\mu\oplus\nu)}$. For the last term, we only need it when applied to a harmonic form and also it should not be orthogonal to a harmonic form.

We now calculate the three terms. For the first term, we begin by applying the Laplace operator on $\SetH$ to the expression.
\begin{align*}
 \Delta &\frac{d^2}{d\epsilon_1 d\bar\epsilon_2}\vert_{\epsilon=0}(\overline{\Phi_2^{\epsilon(\mu\oplus\nu)}})^T\Phi_2^{\epsilon(\mu\oplus\nu)}\\ &=y^2\bar\partial\partial  \frac{d^2}{d\epsilon_1 d\bar\epsilon_2}\vert_{\epsilon=0}(\overline{\Phi_2^{\epsilon(\mu\oplus\nu)}})^T\Phi_2^{\epsilon(\mu\oplus\nu)}\\ &= \frac{d^2}{d\epsilon_1 d\bar\epsilon_2}\vert_{\epsilon=0}y^2\bar\partial\partial((\overline{\Phi_+^{\epsilon(\mu\oplus\nu)}\Phi_-^{\epsilon\mu}})^T\Phi_+^{\epsilon(\mu\oplus\nu)}\Phi_-^{\epsilon\mu})\\ &=\frac{d^2}{d\epsilon_1 d\bar\epsilon_2}\vert_{\epsilon=0}y^2\big(\overline{\bar\partial\Phi_-^{\epsilon\mu}}^T\overline{\partial\Phi_+^{\epsilon(\mu\oplus\nu)}}^T\Phi_+^{\epsilon(\mu\oplus\nu)}\Phi_-^{\epsilon\mu}+\overline{\bar\partial\Phi_-^{\epsilon\mu}}^T\overline{\Phi_+^{\epsilon(\mu\oplus\nu)}}^T\bar\partial\Phi_+^{\epsilon(\mu\oplus\nu)}\Phi_-^{\epsilon\mu}\\ & \hspace{1cm}+\overline{\bar\partial\Phi_-^{\epsilon\mu}}^T\overline{\Phi_+^{\epsilon(\mu\oplus\nu)}}^T\Phi_+^{\epsilon(\mu\oplus\nu)}\bar\partial \Phi_-^{\epsilon\mu}+ \overline{\Phi_-^{\epsilon\mu}}^T\overline{\partial\Phi_+^{\epsilon(\mu\oplus\nu)}}^T\partial\Phi_+^{\epsilon(\mu\oplus\nu)}\Phi_-^{\epsilon\mu}\\ & \hspace{1cm}+\overline{\Phi_-^{\epsilon\mu}}^T\overline{\Phi_+^{\epsilon(\mu\oplus\nu)}}^T\partial\Phi_+^{\epsilon(\mu\oplus\nu)}\bar\partial \Phi_-^{\epsilon\mu}+\overline{\Phi_-^{\epsilon\mu}}^T\overline{\Phi_+^{\epsilon(\mu\oplus\nu)}}^T\bar\partial \partial\Phi_+^{\epsilon(\mu\oplus\nu)} \Phi_-^{\epsilon\mu}\\ &\hspace{1cm}+\overline{\Phi_-^{\epsilon\mu}}^T\overline{\bar\partial \partial\Phi_+^{\epsilon(\mu\oplus\nu)}}^T\Phi_+^{\epsilon(\mu\oplus\nu)} \Phi_-^{\epsilon\mu}+\overline{\Phi_-^{\epsilon\mu}}^T\overline{\bar\partial \Phi_+^{\epsilon(\mu\oplus\nu)}}^T\Phi_+^{\epsilon(\mu\oplus\nu)} \bar\partial\Phi_-^{\epsilon\mu}  \big).
\end{align*}
For all the terms where two different factors are differentiated we are only able to match the $\epsilon$-derivatives in one way that is nonzero. We also have that $\bar\partial \partial\Phi_+^{\epsilon(\mu\oplus\nu)}=\partial \epsilon\mu \partial\Phi_+^{\epsilon(\mu\oplus\nu)}$, and so we need to derive it with respect to $\epsilon$ and $\bar\epsilon$ to get a nonzero contribution. For the same reason $\bar\partial \Phi_+^{\epsilon(\mu\oplus\nu)}$ needs to be differentiated twice to be nonzero. Since $\bar\partial \Phi_+^{\epsilon(\mu\oplus\nu)}$ is always paired with another term, we need to differentiate these terms  and hence they will not contribute, thus we get that
\begin{align*}
  \Delta_0 \frac{d^2}{d\epsilon_1 d\bar\epsilon_2}\vert_{\epsilon=0}(\overline{\Phi_2^{\epsilon(\mu\oplus\nu)}})^T\Phi_2^{\epsilon(\mu\oplus\nu)}&=y^2(\overline{(\nu_2)}^T\overline{(-\bar\nu_1^T)}^T+0+\overline{(\nu_2)}^T\nu_1+\overline{(-\bar\nu_1)^T}^T\\ &\hspace{1cm}\cdot(-\bar\nu_2^T)+(-\bar\nu_2)^T\nu_1-\mu_1\bar\nu_2^T-\bar\mu_2\nu_1+0)\\ &= y^2([\nu_1,\bar\nu_2^T]-\partial\mu_1\bar\nu_2^T-\partial \bar\mu_2\nu_1).
\end{align*}
We conclude that $$\frac{d^2}{d\epsilon_1 d\bar\epsilon_2}\vert_{\epsilon=0}(\overline{\Phi_2^{\epsilon(\mu\oplus\nu)}})^T\Phi_2^{\epsilon(\mu\oplus\nu)}=\Delta^{-1}_0((-\star)\text{ad}\nu_2\star\nu_1-\star\partial\mu_1\bar\nu_2^T-\star\bar\partial\bar\mu_2\nu_1)+cI$$ for some constant $c$. Since the kernel of $\Delta_0$ is the constant multiples of $I$. In what remains this term will not contribute, as we will be looking at $\text{ad}(\frac{d^2}{d\epsilon_1 d\bar\epsilon_2}\vert_{\epsilon=0}(\overline{\Phi_2^{\epsilon(\mu\oplus\nu)}})^T\Phi_2^{\epsilon(\mu\oplus\nu)})$.

Next we calculate the second term $\frac{d^2}{d\epsilon_1 d\bar\epsilon_2}\vert_{\epsilon=0}(\Phi_2^{\epsilon(\mu\oplus\nu)} )^{-1}\partial\Phi_2^{\epsilon(\mu\oplus\nu)}$. The calculation follows directly from the previous computation.
\begin{align*}
 \bar\partial\Delta^{-1}_0&((-\star)\text{ad}\nu_2\star\nu_1-\star\mu_1\bar\nu_2^T-\star\bar\mu_2\nu_1)=\bar\partial\frac{d^2}{d\epsilon_1 d\bar\epsilon_2}\vert_{\epsilon=0}(\overline{\Phi_2^{\epsilon(\mu\oplus\nu)}})^T\Phi_2^{\epsilon(\mu\oplus\nu)}\\ &=\frac{d^2}{d\epsilon_1 d\bar\epsilon_2}\vert_{\epsilon=0}(\overline{\Phi_2^{\epsilon(\mu\oplus\nu)}})^T\bar\partial\Phi_2^{\epsilon(\mu\oplus\nu)}+\frac{d^2}{d\epsilon_1 d\bar\epsilon_2}\vert_{\epsilon=0}(\overline{\partial\Phi_2^{\epsilon(\mu\oplus\nu)}})^T\Phi_2^{\epsilon(\mu\oplus\nu)}\\ &=\frac{d^2}{d\epsilon_1 d\bar\epsilon_2}\vert_{\epsilon=0}(\overline{\Phi_2^{\epsilon(\mu\oplus\nu)}})^T(\Phi_2^{\epsilon(\mu\oplus\nu)} \epsilon\nu+\epsilon\mu\partial\Phi_2^{\epsilon(\mu\oplus\nu)})\\ & \hspace{1cm}+\frac{d^2}{d\epsilon_1 d\bar\epsilon_2}\vert_{\epsilon=0}(\overline{(\Phi_2^{\epsilon(\mu\oplus\nu)})^{-1}\partial\Phi_2^{\epsilon(\mu\oplus\nu)}})^T\overline{\Phi_2^{\epsilon(\mu\oplus\nu)}}^T\Phi_2^{\epsilon(\mu\oplus\nu)}\\ &=-\mu_1\bar\nu_2^T+\frac{d^2}{d\epsilon_1 d\bar\epsilon_2}\vert_{\epsilon=0}(\overline{(\Phi_2^{\epsilon(\mu\oplus\nu)})^{-1}\partial\Phi_2^{\epsilon(\mu\oplus\nu)}})^T,
\end{align*}
since $\frac{d}{d\epsilon}\vert_{\epsilon=0}(\overline{\Phi_2^{\epsilon(\mu\oplus\nu)}})^T\Phi_2^{\epsilon(\mu\oplus\nu)}=0$.

Finally we need to calculate the third term $$\frac{d^2}{d\epsilon_1 d\bar\epsilon_2}\vert_{\epsilon=0} \text{Ad}(\Phi_2^{\epsilon(\mu\oplus\nu)})^{-1}(\Phi_1^{\epsilon\mu})_* P^{0,1}(\Phi_1^{\epsilon\mu})_*^{-1} \text{Ad}\Phi_2^{\epsilon(\mu\oplus\nu)},$$ but only where both $\bar\partial$ and $\bar\partial^*$  in $\bar\partial\Delta_0^{-1}\bar\partial^*$ has been differentiated. This is simplified by the fact that $\bar\partial$  only depending on $\epsilon$ and not $\bar\epsilon$ (see (\ref{eq:barparvarunicoord})). Using this and (\ref{eq:delbarstjerneunicoord}) we have that
\begin{align*}
  \frac{d}{d\epsilon_1}\vert_{\epsilon=0} \text{Ad}(\Phi_2^{\epsilon(\mu\oplus\nu)})&^{-1} (\Phi_1^{\epsilon\mu})_*\bar\partial(\Phi_1^{\epsilon\mu})_*^{-1} \text{Ad}\Phi_2^{\epsilon(\mu\oplus\nu)}
\Delta_0^{-1}
 \\ &\frac{d}{d \bar\epsilon_2}\vert_{\epsilon=0} \text{Ad}(\Phi_2^{\epsilon(\mu\oplus\nu)})^{-1}(\Phi_1^{\epsilon\mu})_* \bar\partial^* (\Phi_1^{\epsilon\mu})_*^{-1}\text{Ad}\Phi_2^{\epsilon(\mu\oplus\nu)}\\ &=(-\mu_1\partial+\text{ad}\nu_1)\Delta_0^{-1}(\partial^*\bar\mu_2-\star\text{ad}\nu_2\star).
\end{align*}
Now we are ready to prove that
\begin{thm}
\label{thm:unviersalsecondordvar}
  Consider the second variation of the metric in the coordinates on the universal moduli space of pairs of a Riemann surface and a holomorphic bundle on it. Then we have this second variation at the center is
  \begin{align*}
\left.    \frac{d^2}{d\epsilon_1d\bar\epsilon_2}\right\vert_{\epsilon=0}g_{\epsilon{(\mu\oplus\nu)}}(\mu_3\oplus\nu_3,&\mu_4\oplus\nu_4)\\ &=-i \int_\Sigma \tr((-\mu_1\partial+\text{ad}\nu_1)\Delta_0^{-1}(\partial^*\bar\mu_2-\star\text{ad}\nu_2\star)\nu_3 \wedge\bar\nu_4^T\\ &-i \int_\Sigma \tr(\text{ad}\Delta^{-1}_0((-\star)\text{ad}\nu_2\star\nu_1-\star(\partial\mu_1\bar\nu_2^T)-\star(\bar\partial\bar\mu_2\nu_1))\nu_3\wedge\bar\nu_4^T)\\ &+i\int_\Sigma \mu_1\bar\mu_2\tr\nu_3\wedge\bar\nu_4^T\\ &-i\int_\Sigma \tr(\text{ad}\nu_1+\mu_1\partial)\Delta_o^{-1}\bar\partial^* \mu_3\bar\nu_2^T\wedge\bar\nu_4^T
\\ &-i\int_\Sigma \tr\mu_3(\partial\Delta_0^{-1}(\star[\star v_1\nu_2]-\star(\partial\mu_1\bar\nu_2^T)-\star(\bar\partial\bar\mu_2\nu_1))\wedge\bar\nu_4^T\\ &-i\int_\Sigma \tr \bar\mu_2\mu_3\nu_1\wedge\overline{\nu_4}^T \\ &-i\int_\Sigma \tr\bar\partial\Delta_o^{-1}(-\star\text{ad}\nu_2\star+\partial^*\mu_2) \nu_3\wedge\bar\mu_4\nu_1
\\ &-i\int_\Sigma \tr\nu_3\wedge \overline{\mu_4(\partial\Delta_0^{-1}(\star[\star v_2\nu_1]-\star(\partial\mu_2\bar\nu_1^T)-\star(\bar\partial\bar\mu_1\nu_2))}^T\\ &-i\int_\Sigma \tr\nu_3\wedge\overline{\bar\mu_1\mu_4\nu_2}^T-i\int_\Sigma \tr\mu_3\nu_1\wedge\overline{\mu_4\nu_2}^T.
  \end{align*}
\end{thm}
\begin{proof}
  Since we already have computed all the ingredients, we gather the results here.
\begin{align*}
 -i\frac{d^2}{d\epsilon_1 d\bar\epsilon_2}\vert_{\epsilon=0}& \int_\Sigma \tr P^{0,1}(\Phi_1^{\epsilon(\mu\oplus\nu)})_*^{-1}\text{Ad}(\Phi_2^{\epsilon(\mu\oplus\nu)})\nu_3\\ &\hspace{1.5cm}\wedge \overline{P^{0,1}((\Phi_1^{\epsilon(\mu\oplus\nu)})_*^{-1}\text{Ad}\Phi_2^{\epsilon(\mu\oplus\nu)})\nu_4}^T\\ &=-i\frac{d^2}{d\epsilon_1 d\bar\epsilon_2}\vert_{\epsilon=0} \int_\Sigma \tr \text{Ad}(\overline{\Phi_2^{\epsilon(\mu\oplus\nu)}}^T\Phi_2^{\epsilon(\mu\oplus\nu)})(1-\vert\epsilon\mu\vert^2))\text{Ad}(\Phi_2^{\epsilon(\mu\oplus\nu)})^{-1}\\ &\hspace{1.5cm}(\Phi_1^{\epsilon(\mu\oplus\nu)})_*P^{0,1}(\Phi_1^{\epsilon(\mu\oplus\nu)})_*^{-1}\text{Ad}(\Phi_2^{\epsilon(\mu\oplus\nu)})\nu_3\wedge  \overline{\nu_4}^T\\ &= -i\int_\Sigma \tr \frac{d^2}{d\epsilon_1 d\bar\epsilon_2}\vert_{\epsilon=0}\text{Ad}(\overline{\Phi_2^{\epsilon(\mu\oplus\nu)}}^T\Phi_2^{\epsilon(\mu\oplus\nu)})\nu_3\wedge  \overline{\nu_4}^T\\ &-i \int_\Sigma \tr \frac{d^2}{d\epsilon_1 d\bar\epsilon_2}\vert_{\epsilon=0}\text{Ad}(\Phi_2^{\epsilon(\mu\oplus\nu)})^{-1}(\Phi_1^{\epsilon(\mu\oplus\nu)})_*\\ & \hspace{2cm}P^{0,1}(\Phi_1^{\epsilon(\mu\oplus\nu)})_*^{-1}\text{Ad}(\Phi_2^{\epsilon(\mu\oplus\nu)})\nu_3\wedge   \overline{\nu_4}^T\\ &-i\int_\Sigma \frac{d^2}{d\epsilon_1 d\bar\epsilon_2}\vert_{\epsilon=0}(1-\vert\epsilon\mu\vert^2)\tr\nu_3\wedge\bar\nu_4^T\\ &= -i\int_\Sigma \tr(\text{ad}(\Delta^{-1}_0((-\star)\text{ad}\nu_2\star\nu_1-\star(\partial\mu_1\bar\nu_2^T)-\star(\bar\partial\bar\mu_2\nu_1)))\nu_3\wedge\bar\nu_4^T)\\ &-i\int_\Sigma \tr((-\mu_1\partial+\text{ad}\nu_1)\Delta_0^{-1}(\partial^*\bar\mu_2-\star\text{ad}\nu_2\star)\nu_3 \wedge\bar\nu_4^T\\ &+i\int_\Sigma \mu_1\bar\mu_2\tr\nu_3\wedge\bar\nu_4^T.
\end{align*}
Now for the second term we have that
\begin{align*}
\frac{d^2}{d\epsilon_1 d\bar\epsilon_2}\vert_{\epsilon=0}-i \int_\Sigma &\tr P^{0,1}(\Phi_1^{\epsilon(\mu\oplus\nu)})_*^{-1}\text{Ad}(\Phi_2^{\epsilon(\mu\oplus\nu)})\mu_3(\Phi_2^{\epsilon(\mu\oplus\nu)})^{-1}\partial\Phi_2^{\epsilon(\mu\oplus\nu)}\\ &\hspace{6cm}\wedge \overline{P^{0,1}((\Phi_1^{\epsilon(\mu\oplus\nu)})_*^{-1}\text{Ad}\Phi_2^{\epsilon(\mu\oplus\nu)})\nu_4}^T\\ &=-i \int_\Sigma \tr P^{0,1}\mu_3 \frac{d^2}{d\epsilon_1 d\bar\epsilon_2}\vert_{\epsilon=0}(\Phi_2^{\epsilon(\mu\oplus\nu)})^{-1}\partial\Phi_2^{\epsilon(\mu\oplus\nu)}\wedge \overline{\nu_4}^T\\ &-i \int_\Sigma \tr\frac{d}{d\epsilon_1}\vert_{\epsilon=0}\text{Ad}(\Phi_2^{\epsilon(\mu\oplus\nu)})^{-1}(\Phi_1^{\epsilon(\mu\oplus\nu)})_*P^{0,1}(\Phi_1^{\epsilon(\mu\oplus\nu)})_*^{-1}\\ &\hspace{3cm}\text{Ad}(\Phi_2^{\epsilon(\mu\oplus\nu)})\mu_3\frac{d}{d\bar\epsilon_2}\vert_{\epsilon=0}(\Phi_2^{\epsilon(\mu\oplus\nu)})^{-1}\partial\Phi_2^{\epsilon(\mu\oplus\nu)}\wedge \overline{\nu_4}^T\\ &= -i \int_\Sigma \tr P^{0,1}\mu_3 (
\partial\Delta^{-1}_0((-\star)\text{ad}\nu_2\star\nu_1-\star(\bar\partial\bar\mu_2\nu_1)-\star(\partial\mu_1\bar\nu_2^T))+\bar\mu_2 \nu_1)
\wedge \overline{\nu_4}^T\\ &-i \int_\Sigma \tr (-\mu_1\partial+\text{ad}\nu_1)\Delta_0^{-1}\bar\partial^*\mu_3\bar\nu_2^T\wedge \overline{\nu_4}^T.
\end{align*}
And similarly
\begin{align*}
-i\int_\Sigma &\tr P^{0,1}(\Phi_1^{\epsilon(\mu\oplus\nu)})_*^{-1}\text{Ad}(\Phi_2^{\epsilon(\mu\oplus\nu)})\nu_3\\ &\hspace{3cm}\wedge \overline{P^{0,1}(\Phi_1^{\epsilon(\mu\oplus\nu)})_*^{-1}(\text{Ad}\Phi_2^{\epsilon(\mu\oplus\nu)})\mu_4(\Phi_2^{\epsilon(\mu\oplus\nu)})\partial\Phi_2^{\epsilon(\mu\oplus\nu)}}^T\\ &= -i \int_\Sigma \tr \nu_3
\wedge \overline{P^{0,1}\mu_4 (
\partial\Delta^{-1}_0((-\star)\text{ad}\nu_1\star\nu_2-\star(\bar\partial\bar\mu_1\nu_2)-\star(\partial\mu_2\bar\nu_1^T))+\bar\mu_1 \nu_2)}^T\\ &-i \int_\Sigma \tr \bar\partial\Delta_0^{-1}(-\partial^*\mu_2-\star\text{ad}\nu_2\star)\nu_3\wedge \overline{\mu_4}{\nu_1}.
            \end{align*}
Finally there is not much choice in how to differentiate the following term
            \begin{align*}-i \frac{d^2}{d\epsilon_1 d\bar\epsilon_2}\vert_{\epsilon=0}\int_\Sigma &\tr P^{0,1}(\Phi_1^{\epsilon(\mu\oplus\nu)})_*^{-1}\text{Ad}(\Phi_2^{\epsilon(\mu\oplus\nu)})\mu_3(\Phi_2^{\epsilon(\mu\oplus\nu)})^{-1}\partial\Phi_2^{\epsilon(\mu\oplus\nu)} \\ &\hspace{.5cm} \wedge \overline{P^{0,1}((\Phi_1^{\epsilon(\mu\oplus\nu)})_*^{-1}\text{Ad}\Phi_2^{\epsilon(\mu\oplus\nu)})\mu_4(\Phi_2^{\epsilon(\mu\oplus\nu)})^{-1}\partial\Phi_2^{\epsilon(\mu\oplus\nu)}}^T\\=-i \int_\Sigma &\tr P^{0,1}(\mu_3\frac{d}{d\bar\epsilon_2}\vert_{\epsilon=0}\partial\Phi_2^{\epsilon(\mu\oplus\nu)}  \wedge \overline{P^{0,1}\mu_4\frac{d}{d\bar\epsilon_1}\vert_{\epsilon=0}\partial\Phi_2^{\epsilon(\mu\oplus\nu)}}^T\\ =-i\int_\Sigma&\tr \mu_3\bar\nu_2^T\wedge\bar\mu_4\nu_1.
\end{align*}
Collect all these results and we have the conclusion.
\end{proof}
\subsection{The Variation of the Metric in Fibered Coordinates}

Now for the fibered coordinates we can do the same computations. From the calculation of the Kodaira-Spencer map (Proposition~\ref{prop:KSfib}) we know, that the metric in the moduli space of bundles direction is given by
\begin{align*}
  &g_{VB}^{\epsilon\nu^{\epsilon\mu}}(\mu_1\oplus\nu_1,\mu_2\oplus\nu_2)=-i\int_\Sigma P_{\epsilon\nu^{\epsilon\mu}}^{0,1}\text{Ad}f^{\epsilon\nu^{\epsilon\mu}}\nu_1^{\epsilon\mu}\wedge\overline{P_{\epsilon\nu^{\epsilon\mu}}^{0,1}\text{Ad}f^{\epsilon\nu^{\epsilon\mu}}\nu_2^{\epsilon\mu}}^T\\ &-i\int_\Sigma P_{\epsilon\nu^{\epsilon\mu}}^{0,1}\text{Ad}f^{\epsilon\nu^{\epsilon\mu}}\nu_1^{\epsilon\mu}\wedge \overline{P_{\epsilon\nu^{\epsilon\mu}}^{0,1}\text{Ad}(f^{\epsilon\nu^{\epsilon\mu}})(\mu_2^{\epsilon\mu}(f^{\epsilon\nu^{\epsilon\mu}})^{-1}\partial f^{\epsilon\nu^{\epsilon\mu}})}^T\\ &-i\int_\Sigma P_{\epsilon\nu^{\epsilon\mu}}^{0,1}\text{Ad}f^{\epsilon\nu^{\epsilon\mu}}\nu_1^{\epsilon\mu}\\ &\wedge\overline{P_{\epsilon\nu^{\epsilon\mu}}^{0,1}\text{Ad}(f^{\epsilon\nu^{\epsilon\mu}})((\Phi_1^{\epsilon\mu})_*^{-1}(1-\vert \epsilon\mu\vert^2)\frac{d}{dt}\vert_{t=0}(\Phi_1^{\epsilon\mu+t\mu_2})P^{0,1}(\Phi_1^{\epsilon\mu+t\mu_2})^{-1}\nu)}^T
\\ &-i\int_\Sigma P_{\epsilon\nu^{\epsilon\mu}}^{0,1}\text{Ad}(f^{\epsilon\nu^{\epsilon\mu}})(\mu_1^{\epsilon\mu}(f^{\epsilon\nu^{\epsilon\mu}})^{-1}\partial f^{\epsilon\nu^{\epsilon\mu}})
\wedge\overline{P_{\epsilon\nu^{\epsilon\mu}}^{0,1}\text{Ad}f^{\epsilon\nu^{\epsilon\mu}}\nu_2^{\epsilon\mu}}^T\\ &-i\int_\Sigma P_{\epsilon\nu^{\epsilon\mu}}^{0,1}\text{Ad}(f^{\epsilon\nu^{\epsilon\mu}})(\mu_1^{\epsilon\mu}(f^{\epsilon\nu^{\epsilon\mu}})^{-1}\partial f^{\epsilon\nu^{\epsilon\mu}})
\wedge \overline{P_{\epsilon\nu^{\epsilon\mu}}^{0,1}\text{Ad}(f^{\epsilon\nu^{\epsilon\mu}})(\mu_2^{\epsilon\mu}(f^{\epsilon\nu^{\epsilon\mu}})^{-1}\partial f^{\epsilon\nu^{\epsilon\mu}})}^T\\ &-i\int_\Sigma P_{\epsilon\nu^{\epsilon\mu}}^{0,1}\text{Ad}(f^{\epsilon\nu^{\epsilon\mu}})(\mu_1^{\epsilon\mu}(f^{\epsilon\nu^{\epsilon\mu}})^{-1}\partial f^{\epsilon\nu^{\epsilon\mu}})
\\ &\wedge\overline{P_{\epsilon\nu^{\epsilon\mu}}^{0,1}\text{Ad}(f^{\epsilon\nu^{\epsilon\mu}})((\Phi_1^{\epsilon\mu})_*^{-1}(1-\vert \epsilon\mu\vert^2)\frac{d}{dt}\vert_{t=0}(\Phi_1^{\epsilon\mu+t\mu_2})P^{0,1}(\Phi_1^{\epsilon\mu+t\mu_2})^{-1}\nu)}^T
\\ &-i\int_\Sigma P_{\epsilon\nu^{\epsilon\mu}}^{0,1}\text{Ad}(f^{\epsilon\nu^{\epsilon\mu}})((\Phi_1^{\epsilon\mu})_*^{-1}(1-\vert \epsilon\mu\vert^2)\frac{d}{dt}\vert_{t=0}(\Phi_1^{\epsilon\mu+t\mu_1})P^{0,1}(\Phi_1^{\epsilon\mu+t\mu_1})^{-1}\nu)\\ &\hspace{9.5cm}\wedge\overline{P_{\epsilon\nu^{\epsilon\mu}}^{0,1}\text{Ad}f^{\epsilon\nu^{\epsilon\mu}}\nu_2^{\epsilon\mu}}^T\\ &-i\int_\Sigma P_{\epsilon\nu^{\epsilon\mu}}^{0,1}\text{Ad}(f^{\epsilon\nu^{\epsilon\mu}})((\Phi_1^{\epsilon\mu})_*^{-1}(1-\vert \epsilon\mu\vert^2)\frac{d}{dt}\vert_{t=0}(\Phi_1^{\epsilon\mu+t\mu_1})P^{0,1}(\Phi_1^{\epsilon\mu+t\mu_1})^{-1}\nu)\\ &\hspace{6cm}\wedge \overline{P_{\epsilon\nu^{\epsilon\mu}}^{0,1}\text{Ad}(f^{\epsilon\nu^{\epsilon\mu}})(\mu_2^{\epsilon\mu}(f^{\epsilon\nu^{\epsilon\mu}})^{-1}\partial f^{\epsilon\nu^{\epsilon\mu}})}^T\\ &-i\int_\Sigma P_{\epsilon\nu^{\epsilon\mu}}^{0,1}\text{Ad}(f^{\epsilon\nu^{\epsilon\mu}})((\Phi_1^{\epsilon\mu})_*^{-1}(1-\vert \epsilon\mu\vert^2)\frac{d}{dt}\vert_{t=0}(\Phi_1^{\epsilon\mu+t\mu_1})P^{0,1}(\Phi_1^{\epsilon\mu+t\mu_1})^{-1}\nu)\\ &\wedge\overline{P_{\epsilon\nu^{\epsilon\mu}}^{0,1}\text{Ad}(f^{\epsilon\nu^{\epsilon\mu}})((\Phi_1^{\epsilon\mu})_*^{-1}(1-\vert \epsilon\mu\vert^2)\frac{d}{dt}\vert_{t=0}(\Phi_1^{\epsilon\mu+t\mu_2})P^{0,1}(\Phi_1^{\epsilon\mu+t\mu_2})^{-1}\nu)}^T.
\end{align*}
While these nine terms look intimidating, we can discard three of the terms, because $P_{\epsilon\nu^{\epsilon\mu}}^{0,1}\text{Ad}(f^{\epsilon\nu^{\epsilon\mu}})((\Phi_1^{\epsilon\mu})_*^{-1}(1-\vert \epsilon\mu\vert^2)\frac{d}{dt}\vert_{t=0}(\Phi_1^{\epsilon\mu+t\mu_2})P^{0,1}(\Phi_1^{\epsilon\mu+t\mu_2})^{-1}\nu)$ vanishes to second-order and $P_{\epsilon\nu^{\epsilon\mu}}^{0,1}\text{Ad}(f^{\epsilon\nu^{\epsilon\mu}})(\mu_2^{\epsilon\mu}(f^{\epsilon\nu^{\epsilon\mu}})^{-1}\partial f^{\epsilon\nu^{\epsilon\mu}})$ vanishes to first-order, so terms containing both kind of factors or only the first kind of factors will vanish to higher order, than we are interested in. Now the first variation will be the same as in Section \ref{sec:metvaruniv}, but to calculate it we will have to work with slightly different expressions.

First we consider
\begin{align*}
  \frac{d}{d\epsilon}\vert_{\epsilon=0} ((\overline{f^{\epsilon\nu^\epsilon\mu}})^Tf^{\epsilon\nu^\epsilon\mu})\circ\Phi_1^{\epsilon\mu}=0
\end{align*}
and
\begin{align*}
  \frac{d}{d\bar\epsilon}\vert_{\epsilon=0} ((\overline{f^{\epsilon\nu^\epsilon\mu}})^Tf^{\epsilon\nu^\epsilon\mu})\circ\Phi_1^{\epsilon\mu}=0.
\end{align*}
Both of these follow from the computations in \cite{ZTVB}, where it was shown that $\frac{d}{d\bar\epsilon}\vert_{\epsilon=0} ((\overline{f^{\epsilon\nu}})^Tf^{\epsilon\nu})=0$. Now composing with $\nu\to\nu^{\epsilon\mu}$ won't change it, and if we differentiate $\Phi_1^{\epsilon\mu}$ then we can set $\epsilon=0$ in the rest of the terms and calculate $\frac{d}{d\epsilon} I\circ\Phi_1^{\epsilon\mu}=0$. Now for a projection, the first derivative will either have harmonic forms in it's kernel or the image is in the orthogonal complement, hence the only contributions are from the terms
\begin{align*}
  \int_\Sigma P_{\epsilon\nu^{\epsilon\mu}}^{0,1}\text{Ad}f^{\epsilon\nu^{\epsilon\mu}}\nu_1^{\epsilon\mu} \overline{P_{\epsilon\nu^{\epsilon\mu}}^{0,1}\text{Ad}(f^{\epsilon\nu^{\epsilon\mu}})(\mu_2^{\epsilon\mu}(f^{\epsilon\nu^{\epsilon\mu}})^{-1}\partial f^{\epsilon\nu^{\epsilon\mu}})}^T
\end{align*}
and
\begin{align*}
\int_\Sigma P_{\epsilon\nu^{\epsilon\mu}}^{0,1}\text{Ad}(f^{\epsilon\nu^{\epsilon\mu}})(\mu_1^{\epsilon\mu}(f^{\epsilon\nu^{\epsilon\mu}})^{-1}\partial f^{\epsilon\nu^{\epsilon\mu}})
\overline{P_{\epsilon\nu^{\epsilon\mu}}^{0,1}\text{Ad}f^{\epsilon\nu^{\epsilon\mu}}\nu_2^{\epsilon\mu}}^T.
\end{align*}
And so we have, completely analogues to the previous section the following lemma.
\begin{lemma}
  In the fibered coordinates around $(X,E)$  we have that:
  \begin{align*}
    \frac{d}{d\epsilon}\vert_{\epsilon=0}g^{VB}_{\epsilon\nu^{\epsilon\mu}}(&\mu_1\oplus\nu_1,\mu_2\oplus\nu_2)=i\int_X\bar\mu_2\tr(\nu_1\nu)\\  \frac{d}{d\bar\epsilon}\vert_{\epsilon=0}g^{VB}_{\epsilon\nu^{\epsilon\mu}}(&\mu_1\oplus\nu_1,\mu_2\oplus\nu_2)=i\int_X\mu_1\tr(\bar\nu^T \bar\nu_2^T)
  \end{align*}
\end{lemma}
Now for the second variation of the metric we need to calculate the two terms $$\frac{d^2}{d\epsilon_1d\bar\epsilon_2}\vert_{\epsilon=0}((\overline{f^{\epsilon\nu^{\epsilon\mu}}})^Tf^{\epsilon\nu^{\epsilon\mu}})\circ\Phi_1^{\epsilon\mu}$$ and  $$\frac{d^2}{d\epsilon_1d\bar\epsilon_2}\vert_{\epsilon=0}((f^{\epsilon\nu^{\epsilon\mu}})^{-1}\partial f^{\epsilon\nu^{\epsilon\mu}})\circ\Phi_1^{\epsilon\mu}.$$ We calculate these the same way as we did with the previous set of coordinates.
\begin{align*}
  \Delta_0\frac{d^2}{d\epsilon_1d\bar\epsilon_2}\vert_{\epsilon=0}((\overline{f^{\epsilon\nu^{\epsilon\mu}}})^Tf^{\epsilon\nu^{\epsilon\mu}})\circ\Phi_1^{\epsilon\mu}=y^2\partial\bar\partial\frac{d^2}{d\epsilon_1d\bar\epsilon_2}\vert_{\epsilon=0}((\overline{f^{\epsilon\nu^{\epsilon\mu}}})^Tf^{\epsilon\nu^{\epsilon\mu}})\circ\Phi_1^{\epsilon\mu}\\ =\frac{d^2}{d\epsilon_1d\bar\epsilon_2}\vert_{\epsilon=0}y^2\partial\bar\partial((\overline{f_-^{\epsilon\nu^{\epsilon\mu}}})^T(\overline{f_+^{\epsilon\nu^{\epsilon\mu}}})^Tf_+^{\epsilon\nu^{\epsilon\mu}}f_-^{\epsilon\nu^{\epsilon\mu}})\circ\Phi_1^{\epsilon\mu}.
\end{align*}
Now we use \begin{align*}\bar\partial\partial (h\circ\Phi_1^{\epsilon\mu}) &=\partial\Phi_1^{\epsilon\mu}\partial\bar\Phi_1^{\epsilon\mu}(\partial\partial h)\circ\Phi_1^{\epsilon\mu}
+\bar\partial\bar\Phi_1^{\epsilon\mu}\partial\Phi_1^{\epsilon\mu}(\partial\bar\partial h)\circ\Phi_1^{\epsilon\mu}
\\ &\quad+\bar\partial\Phi_1^{\epsilon\mu}\partial\bar\Phi_1^{\epsilon\mu}(\bar\partial\partial h)\circ\Phi_1^{\epsilon\mu}+\partial\bar\Phi_1^{\epsilon\mu}\bar\partial\bar\Phi_1^{\epsilon\mu}(\bar\partial\bar\partial h) \circ\Phi_1^{\epsilon\mu}\\ &=(\partial\Phi_1^{\epsilon\mu}\epsilon\mu\bar\partial\bar\Phi_1^{\epsilon\mu}(\partial\partial h)\circ\Phi_1^{\epsilon\mu}
+\bar\partial\bar\Phi_1^{\epsilon\mu}\partial\Phi_1^{\epsilon\mu}(\partial\bar\partial h)\circ\Phi_1^{\epsilon\mu}
\\ &\quad+\vert\epsilon\mu\vert^2\partial\Phi_1^{\epsilon\mu}\partial\bar\Phi_1^{\epsilon\mu}(\bar\partial\partial h)\circ\Phi_1^{\epsilon\mu}+\overline{\epsilon\mu}\bar\partial\bar\Phi_1^{\epsilon\mu}\bar\partial\bar\Phi_1^{\epsilon\mu}(\bar\partial\bar\partial h )\circ\Phi_1^{\epsilon\mu},
                \end{align*}
for $h=(\overline{f^{\epsilon\nu^{\epsilon\mu}}})^Tf^{\epsilon\nu^{\epsilon\mu}})$, and since we know that $\partial f^{\epsilon\nu^{\epsilon\mu}}$ and $\bar\partial f^{\epsilon\nu^{\epsilon\mu}}$ vanish to first-order in $\epsilon$, we only have the surviving terms
\begin{align*}
  \Delta_0\frac{d^2}{d\epsilon_1d\bar\epsilon_2}\vert_{\epsilon=0}((\overline{f^{\epsilon\nu^{\epsilon\mu}}})^Tf^{\epsilon\nu^{\epsilon\mu}})\circ\Phi_1^{\epsilon\mu}=y^2(-\partial\mu_1\bar\nu_2^T-\bar\partial\bar\mu_2\nu_1+[\nu_1,\bar\nu_2^T]),
\end{align*}
which is exactly like in the previous case. We proceed on to calculate $\frac{d^2}{d\epsilon_1d\bar\epsilon_2}\vert_{\epsilon=0}((f^{\epsilon\nu^{\epsilon\mu}})^{-1}\partial f^{\epsilon\nu^{\epsilon\mu}})\circ\Phi_1^{\epsilon\mu}$, and so we study
\begin{align*}
\bar\partial \Delta_0^{-1}y^2(-\mu_1\partial\bar\nu_2^T-\bar\mu_2\bar\partial\nu_1+[\nu_1,\bar\nu_2^T])&=  \bar\partial \frac{d^2}{d\epsilon_1d\bar\epsilon_2}\vert_{\epsilon=0}((\overline{f^{\epsilon\nu^{\epsilon\mu}}})^Tf^{\epsilon\nu^{\epsilon\mu}})\circ\Phi_1^{\epsilon\mu} \\ &= \frac{d^2}{d\epsilon_1d\bar\epsilon_2}\vert_{\epsilon=0}(\bar\partial((\overline{f^{\epsilon\nu^{\epsilon\mu}}})^Tf^{\epsilon\nu^{\epsilon\mu}}))\circ\Phi_1^{\epsilon\mu}\bar\partial \bar\Phi_1^{\epsilon\mu}\\ &+\frac{d^2}{d\epsilon_1d\bar\epsilon_2}\vert_{\epsilon=0}(\partial((\overline{f^{\epsilon\nu^{\epsilon\mu}}})^Tf^{\epsilon\nu^{\epsilon\mu}}))\circ\Phi_1^{\epsilon\mu}\bar\partial \Phi_1^{\epsilon\mu}\\ &= \frac{d^2}{d\epsilon_1d\bar\epsilon_2}\vert_{\epsilon=0}(((\overline{\partial f^{\epsilon\nu^{\epsilon\mu}}})^Tf^{\epsilon\nu^{\epsilon\mu}}))\circ\Phi_1^{\epsilon\mu}\bar\partial \bar\Phi_1^{\epsilon\mu}\\ &+\frac{d^2}{d\epsilon_1d\bar\epsilon_2}\vert_{\epsilon=0}(((\overline{\bar\partial f^{\epsilon\nu^{\epsilon\mu}}})^Tf^{\epsilon\nu^{\epsilon\mu}}))\circ\Phi_1^{\epsilon\mu}\bar\partial \Phi_1^{\epsilon\mu}\\ &+ \frac{d^2}{d\epsilon_1d\bar\epsilon_2}\vert_{\epsilon=0}(((\overline{f^{\epsilon\nu^{\epsilon\mu}}})^T \bar\partial f^{\epsilon\nu^{\epsilon\mu}}))\circ\Phi_1^{\epsilon\mu}\bar\partial \bar\Phi_1^{\epsilon\mu}\\ &+\frac{d^2}{d\epsilon_1d\bar\epsilon_2}\vert_{\epsilon=0}(((\overline{f^{\epsilon\nu^{\epsilon\mu}}})^T \partial f^{\epsilon\nu^{\epsilon\mu}}))\circ\Phi_1^{\epsilon\mu}\bar\partial \Phi_1^{\epsilon\mu}.
\end{align*}
Here the second and fourth term cancel, as is seen by using that two of the factors vanish to first-order in $\epsilon$, which then give $\mu_1\bar\nu_2^T$ and $-\mu_1\bar\nu_2^T$ respectively.
\begin{align*}
 \frac{d^2}{d\epsilon_1d\bar\epsilon_2}\vert_{\epsilon=0}&(((\overline{\partial f^{\epsilon\nu^{\epsilon\mu}}})^Tf^{\epsilon\nu^{\epsilon\mu}}))\circ\Phi_1^{\epsilon\mu}\bar\partial \bar\Phi_1^{\epsilon\mu} \\ &\quad + \frac{d^2}{d\epsilon_1d\bar\epsilon_2}\vert_{\epsilon=0}(((\overline{f^{\epsilon\nu^{\epsilon\mu}}})^T \bar\partial f^{\epsilon\nu^{\epsilon\mu}}))\circ\Phi_1^{\epsilon\mu}\bar\partial \bar\Phi_1^{\epsilon\mu}\\ &=
\frac{d^2}{d\epsilon_1d\bar\epsilon_2}\vert_{\epsilon=0}((\overline{(f^{\epsilon\nu^{\epsilon\mu}})^{-1}\partial f^{\epsilon\nu^{\epsilon\mu}}}^T\overline{f^{\epsilon\nu^{\epsilon\mu}}}^Tf^{\epsilon\nu^{\epsilon\mu}}))\circ\Phi_1^{\epsilon\mu}\bar\partial \bar\Phi_1^{\epsilon\mu} \\ &\quad + \frac{d^2}{d\epsilon_1d\bar\epsilon_2}\vert_{\epsilon=0}(((\overline{f^{\epsilon\nu^{\epsilon\mu}}})^T  f^{\epsilon\nu^{\epsilon\mu}}\epsilon\nu^{\epsilon\mu}))\circ\Phi_1^{\epsilon\mu}\bar\partial \bar\Phi_1^{\epsilon\mu}\\ &=
\frac{d^2}{d\epsilon_1d\bar\epsilon_2}\vert_{\epsilon=0}((\overline{(f^{\epsilon\nu^{\epsilon\mu}})^{-1}\partial f^{\epsilon\nu^{\epsilon\mu}}}^T)\circ\Phi_1^{\epsilon\mu}\bar\partial \bar\Phi_1^{\epsilon\mu} \\ &\quad + \frac{d^2}{d\epsilon_1d\bar\epsilon_2}\vert_{\epsilon=0}(\epsilon\nu^{\epsilon\mu}))\circ\Phi_1^{\epsilon\mu}\bar\partial \bar\Phi_1^{\epsilon\mu}\\ &=
\frac{d^2}{d\epsilon_1d\bar\epsilon_2}\vert_{\epsilon=0}((\overline{(f^{\epsilon\nu^{\epsilon\mu}})^{-1}\partial f^{\epsilon\nu^{\epsilon\mu}}}^T)\circ\Phi_1^{\epsilon\mu}\bar\partial \bar\Phi_1^{\epsilon\mu}+\bar\partial\Delta_0^{-1}\partial^* \bar\mu_2\nu_1.
\end{align*}
Now this is different from the previous coordinates.
We need however to consider two things more. The first is the second variation of the harmonic projection. Since the only relevant part is the contribution where both $\bar\partial$ and $\bar\partial^*$ have been differentiated in $\bar\partial\Delta_0^{-1}\bar\partial^*$ and the coordinates agree to second-order nothing will have changed and we have it gives the following term
\begin{align*}
(-\mu_1\partial+\text{ad}\nu_1)\Delta_0^{-1}(\partial^*\bar\mu_2-\star\text{ad}\nu_2\star).  
\end{align*}

The final term to consider is the new term in the formula for the metric, which is
\begin{align*}
\frac{d^2}{d\epsilon_1d\bar\epsilon_2}\vert_{\epsilon=0}  P_{\epsilon\nu^{\epsilon\mu}}^{0,1}\text{Ad}(f^{\epsilon\nu^{\epsilon\mu}})((\Phi_1^{\epsilon\mu})_*^{-1}(1-\vert \epsilon\mu\vert^2)\frac{d}{dt}\vert_{t=0}(\Phi_1^{\epsilon\mu+t{\tilde\mu}})P^{0,1}(\Phi_1^{\epsilon\mu+t{\tilde\mu}})^{-1}\epsilon\nu),
\end{align*}
as $\frac{d}{dt}\vert_{t=0}(\Phi_1^{\epsilon\mu+t{\tilde\mu}})P^{0,1}(\Phi_1^{\epsilon\mu+t{\tilde\mu}})^{-1}\epsilon\nu)$ vanishes to second-order this has to be differentiated twice
\begin{align*}
 \frac{d}{d\bar\epsilon_2}\vert_{\epsilon=0} \frac{d}{dt}\vert_{t=0}&(\Phi_1^{\epsilon\mu+t{\tilde\mu}})P^{0,1}(\Phi_1^{\epsilon\mu+t{\tilde\mu}})^{-1}\nu_1)\\ &=- (\frac{d}{d\bar\epsilon_2}\vert_{\epsilon=0} \frac{d}{dt}\vert_{t=0}(\Phi_1^{\epsilon\mu+t{\tilde\mu}})\bar\partial (\Phi_1^{\epsilon\mu+t{\tilde\mu}})^{-1}\\ &\hspace{4cm}\Delta_0^{-1} \frac{d}{d\bar\epsilon_2}\vert_{\epsilon=0}\Phi_1^{\epsilon\mu+t{\tilde\mu}})\bar\partial^*(\Phi_1^{\epsilon\mu+t{\tilde\mu}})^{-1}\nu_1)\\\quad&-\bar\partial  \frac{d}{d\bar\epsilon_2}\vert_{\epsilon=0} \frac{d}{dt}\vert_{t=0}(\Phi_1^{\epsilon\mu+t{\tilde\mu}})\Delta_0^{-1}\bar\partial^*(\Phi_1^{\epsilon\mu+t{\tilde\mu}})^{-1})\nu_1\\ &=-\tilde\mu\partial\Delta_0^{-1}\partial^*\bar\mu_2\nu_1-\bar\partial  \frac{d}{d\bar\epsilon_2}\vert_{\epsilon=0} \frac{d}{dt}\vert_{t=0}(\Phi_1^{\epsilon\mu+t{\tilde\mu}})\Delta_0^{-1}\bar\partial^*(\Phi_1^{\epsilon\mu+t{\tilde\mu}})^{-1}\nu_1.
\end{align*}
We are now ready to gather all the contributions in the following theorem.
\begin{thm}
  \label{lem:metvarfibcord}
We have the following for the second variation of the metric at $(X,E)$ in the fibered coordinates:
\begin{align*}
   \frac{d^2}{d\epsilon_1d\bar\epsilon_2}\vert_{\epsilon=0}g^{VB}_{\epsilon\nu^{\epsilon\mu}}(&\mu_3\oplus\nu_3,\mu_4\oplus\nu_4)= \int_\Sigma \tr((-\mu_1\partial+\text{ad}\nu_1)\Delta_0^{-1}((-\star)\text{ad}\nu_2\star+\partial^*\bar\mu_2)\nu_3 )\wedge\bar\nu_4^T\\ &-i \int_\Sigma \tr(\text{ad}\Delta^{-1}_0((-\star)\text{ad}\nu_2\star\nu_1-\star(\partial\mu_1\bar\nu_2^T)-\star(\bar\partial\bar\mu_2\nu_1))\nu_3\wedge\bar\nu_4^T)\\ &-i\int_\Sigma \tr(\text{ad}\nu_1-\mu_1\partial)\Delta_o^{-1}\bar\partial^* \mu_3\bar\nu_2^T\wedge\bar\nu_4^T
\\ &+i\int_\Sigma \mu_1\bar\mu_2\tr\nu_3\wedge\bar\nu_4^T\\ &-i\int_\Sigma \tr\mu_3(\partial\Delta_0^{-1}(\star[\star \nu_1\nu_2]-\star(\partial\mu_1\bar\nu_2^T)-\star(\bar\partial\bar\mu_2\nu_1))\wedge\bar\nu_4^T\\ &+i\int_\Sigma \tr \mu_3 \partial\Delta_0^{-1}\bar\partial^* \mu_1\bar\nu_2^T\wedge\overline{\nu_4}^T \\ &-i\int_\Sigma \tr\bar\partial\Delta_o^{-1}(-\star\text{ad}\nu_2\star+\partial^*\bar\mu_2) \nu_3\wedge(-\bar\mu_4\nu_1)
\\ &-i\int_\Sigma \tr\nu_3\wedge \overline{\mu_4(\partial\Delta_0^{-1}(\star[\star \nu_2\nu_1]-\star(\partial\mu_2\bar\nu_1^T)-\star(\bar\partial\bar\mu_1\nu_2))}^T\\ &+i\int_\Sigma \tr\nu_3\wedge\overline{\mu_4\partial\Delta_0^{-1}\bar\partial^*\mu_2\bar\nu_1^T}^T-i\int_\Sigma \tr\mu_3\nu_1\wedge\overline{\mu_4\nu_2}^T\\ &+i\int_\Sigma \mu_3\partial\Delta_0^{-1}\partial^*\bar\mu_2\nu_1\wedge \bar\nu_4^T+i\int_\Sigma \nu_3\wedge \overline{\mu_4\partial\Delta_0^{-1}\partial^*\bar\mu_1\nu_2}^T
\end{align*}
\end{thm}
Comparing this to the previous coordinates (Theorem~\ref{thm:unviersalsecondordvar}) we see that there are four terms here which we didn't have before and two we no longer have. The new terms are
\begin{align*}
  i\int_\Sigma \mu_3\partial\Delta_0^{-1}\partial^*\bar\mu_2\nu_1\wedge \bar\nu_4^T,\\i\int_\Sigma \nu_3\wedge \overline{\mu_4\partial\Delta_0^{-1}\partial^*\bar\mu_1\nu_2}^T,\\
i\int_\Sigma \tr \mu_3 \partial\Delta_0^{-1}\bar\partial^* \mu_1\bar\nu_2^T\wedge\overline{\nu_4}^T
  \end{align*}
    and
  \begin{align*}
i\int_\Sigma \tr\nu_3\wedge\overline{\mu_4\partial\Delta_0^{-1}\bar\partial^*\mu_2\bar\nu_1^T}^T,
\end{align*}
While the ones we no longer have are
\begin{align*}
-i\int_\Sigma \tr\nu_3\wedge\overline{\mu_1\bar\mu_4\nu_2}^T  \end{align*}
    and
  \begin{align*} -i\int_\Sigma \tr \bar\mu_2\mu_3\nu_1\wedge\overline{\nu_4}^T. 
\end{align*}
Now we have that for $\nu_1=\nu_4$ and $\mu_2=\mu_3$ and the rest $0$ the difference between the two expressions are
\begin{multline*}
  - i\int_\Sigma \mu_3\partial\Delta_0^{-1}\partial^*\bar\mu_2\nu_1\wedge \bar\nu_4^T-i\int_\Sigma \tr \bar\mu_2\mu_3\nu_1\wedge\overline{\nu_4}^T\\=
  - i\int_\Sigma \Delta_0^{-1}\partial^*\bar\mu_2\nu_1\wedge \overline{\partial^*\bar\mu_2\nu_1}^T\rho-i\int_\Sigma \tr \bar\mu_2\mu_2\nu_1\wedge\overline{\nu_1}^T
\end{multline*}
Since $\Delta$ is a positive operator we have that $
 -i \int_\Sigma \Delta_0^{-1}\partial^*\bar\mu_2\nu_1\wedge \overline{\partial^*\bar\mu_2\nu_1}^T\rho\geq0 $ and for obvious reasons $-i\int_\Sigma \tr \bar\mu_2\mu_2\nu_1\wedge\overline{\nu_1}^T>0$, if $\nu_1\neq 0$ and $\mu_2\neq 0$.
 \begin{thm}
\label{thm:differthrid}   The coordinates of \ref{thm:uni coord} and the Fibered coordinates agree to second order, but differ at thrid order in the derivatives at the center point.
 \end{thm}

\end{document}